\documentclass[11pt]{amsart}

\usepackage[left=2.5cm,right=2.5cm,top=3cm,bottom=3cm]{geometry}

\usepackage{amscd,amsfonts,amsmath,amssymb}
\usepackage[toc,page]{appendix}
\usepackage{array}
\usepackage[english]{babel}
\usepackage{bigstrut} \usepackage{bbold}
\usepackage{calc}
\usepackage{caption}
\usepackage{color}
\usepackage{colortbl}
\usepackage{csquotes}
\usepackage{dsfont}
\usepackage{enumerate}
\usepackage{enumitem}
\usepackage{epsfig}
\usepackage{esint}
\usepackage{hyperref}
\usepackage{etoolbox}
\usepackage{float}
\usepackage{graphpap}
\usepackage{graphics}
\usepackage{graphicx}
\usepackage{latexsym}
\usepackage{lipsum}
\usepackage{lmodern}
\usepackage{marginnote}
\usepackage{mathrsfs}
\usepackage{mathtools}
\usepackage{mathabx}        \usepackage{multirow}
\usepackage{pdfpages}
\usepackage{pgfplots}
\usepackage{pxfonts}
\usepackage{rotating}
\usepackage{subcaption}
\usepackage{tabularx,booktabs}
\newcolumntype{Y}{>{\centering\arraybackslash}X}
\usepackage{tikz}
\usetikzlibrary{shapes}
\usetikzlibrary{shapes.misc}

\usepgfplotslibrary{patchplots}
\usetikzlibrary{patterns, positioning, arrows}
\pgfplotsset{compat=1.13}
\tikzset{cross/.style={cross out, draw=black, fill=none, minimum size=2*(#1-\pgflinewidth), inner sep=0pt, outer sep=0pt}, cross/.default={2pt}}

\usepackage{times}
\usepackage{varwidth}
\usepackage{wrapfig}
\usepackage{xspace}

\usepackage[english]{babel}
\pgfplotsset{compat=newest}

\newcommand{\R}{\mathbb{R}}
\renewcommand{\d}{\ \mathrm{d}}

\renewcommand{\d}{{\,\mathrm{d}}}
\newcommand{\norm}[1]{\left\lVert#1\right\rVert}
\newcommand{\restr}[1]{|_{#1}}
\newcommand{\lznormT}[1]{\left\lVert#1\right\rVert_{L^2(T)}}
\newcommand{\lznormO}[1]{\left\lVert#1\right\rVert_{L^2(\omega)}}
\newcommand{\partb}[1]{\partial_{#1} \psi_B}
\newcommand{\parta}[1]{\partial_{#1} \psi_A}

\newcommand{\intomega}[1]{\int_\omega {#1} \d x}

\newcommand{\transRelShapeOpt}[1]{B\left[#1\right]}
\newcommand{\transRelShapeOpth}[1]{B_{h}\left[#1\right]}
\newcommand{\secFF}{II}

\DeclareMathOperator {\tr} {tr}

\newtheorem{definition}{Definition}

\newtheorem{proposition}[definition]{Proposition}
\newtheorem{theorem}[definition]{Theorem}
\newtheorem{corollary}[definition]{Corollary}
\numberwithin{definition}{section}

\makeatletter
\DeclareRobustCommand\onedot{\futurelet\@let@token\@onedot}
\def\@onedot{\ifx\@let@token.\else.\null\fi\xspace}
 
\def\ie{\emph{i.e}\onedot} 
\def\cf{\emph{cf}\onedot}

\def\etal{\emph{et al}\onedot}

\makeatother

\begin{document}

\title{Finite Element Approximation of Large-Scale Isometric Deformations of Parametrized Surfaces}
\author{
Martin Rumpf
\and Stefan Simon 
\and Christoph Smoch
}
\address{Institute for Numerical Simulation, University of Bonn, Endenicher Allee 60, 53115 Bonn, Germany}
\email{martin.rumpf@uni-bonn.de, s6stsimo@uni-bonn.de, s6chsmoc@uni-bonn.de}
\date{\today}
\renewcommand{\subjclassname}{\textup{2010} Mathematics Subject Classification}

\begin{abstract}
In this paper, the numerical approximation of isometric deformations of thin elastic shells is discussed.
To this end, for a thin shell represented by a parametrized surface,
it is shown how to transform the stored elastic energy for an isometric deformation
such that the highest order term is quadratic.
For this reformulated model, existence of optimal isometric deformations is shown.
A finite element approximation is obtained using the Discrete Kirchhoff Triangle (DKT) approach
and the convergence of discrete minimizers to a continuous minimizer is demonstrated.
In that respect, this paper generalizes the results by Bartels
for the approximation of bending isometries of plates.
A Newton scheme is derived to numerically simulate large bending isometries of shells.
The proven convergence properties are experimentally verified
and characteristics of isometric deformations are discussed.
\end{abstract}

\maketitle

\section{Introduction} \label{sec:intro}
We investigate deformations of thin elastic objects and their numerical approximation.
These objects are frequently characterized by a small thickness $\delta > 0$ and a regular and orientable two-dimensional midsurface $\mathcal{M}_A$.
Given an external force $f_A \colon \mathcal{M}_A \to \R^3$ acting on the thin object, equilibrium deformations have been extensively studied in the literature.
In particular, considering the limit of vanishing thickness,
$\Gamma$-convergence allows to express the 3D deformation of the thin object
by a 2D deformation of its midsurface.
In \cite{LeRa95,LeRa96}, Le Dret and Raoult obtained a membrane theory describing tangential distortion on the surface.
In this paper, we focus on a bending theory taking into account isometric deformations.
For such bending isometries, a first $\Gamma$-convergence result was provided by Friesecke~\etal in \cite{FrJaMue02b}
by rigorously deriving Kirchhoff's plate theory from nonlinear three dimensional elasticity.
In this special case of the two-dimensional midsurface being a flat object $\mathcal{M}_A = \omega \times \{ 0 \}$ for some suitable $\omega \subset \R^2$,
smooth isometric deformations are characterized by the global property that they are developable surfaces.
This has been shown for smooth isometries by Hartman and Nirenberg in \cite{HaNi59}. 
Moreover, Hornung~\cite{Ho11} has proven that this result holds true for $H^2$ isometries.
In~\cite{FrJaMo03}, Friesecke~\etal extended the $\Gamma$-convergence result to case of thin elastic shells,
where the corresponding midsurface $\mathcal{M}_A$ is in general allowed to be curved.
More precisely, it was shown that the bending energy depends on the so-called relative shape operator,
which we will detail in the following.

Throughout this paper, we will restrict to parametrized surfaces,
\ie $\mathcal{M}_A = \psi_A(\omega)$ for a bounded and connected Lipschitz domain $\omega \subset \R^2$
and an injective parametrization $\psi_A \in H^3(\omega; \R^3)$.
An external force $f_A \in L^2(\mathcal{M}_A; \R^3)$ acting on the midsurface is given via some $f\in L^2(\omega; \R^3)$ on the parameter domain with $f=f_A \circ \psi_A$.
The deformed midsurface $\mathcal{M}_B = \psi_B(\omega)$ is described by a parametrization $\psi_B \in H^2(\omega; \R^3)$.
The resulting actual deformation $\phi \colon \psi_A(\omega) \to \R^3$ of the thin shell midsurface $\psi_A(\omega)$
is then given by $\phi=\psi_B \circ \psi_A^{-1}$ (cf. Figure~\ref{fig:overview}).
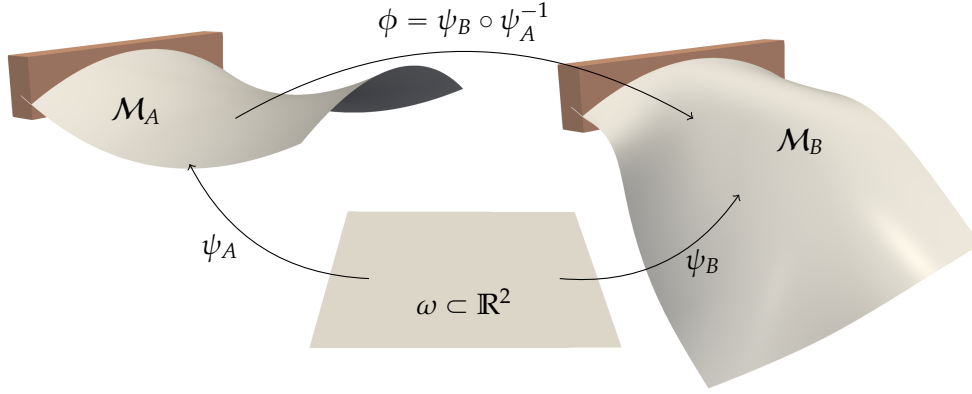
\begin{figure}[htbp]
	\centering
	\resizebox{0.8\textwidth}{!}{
	\begin{tikzpicture}
        \node[inner sep=0pt] (Deformed) at (7,1.2) {\includegraphics[width=0.34\textwidth]{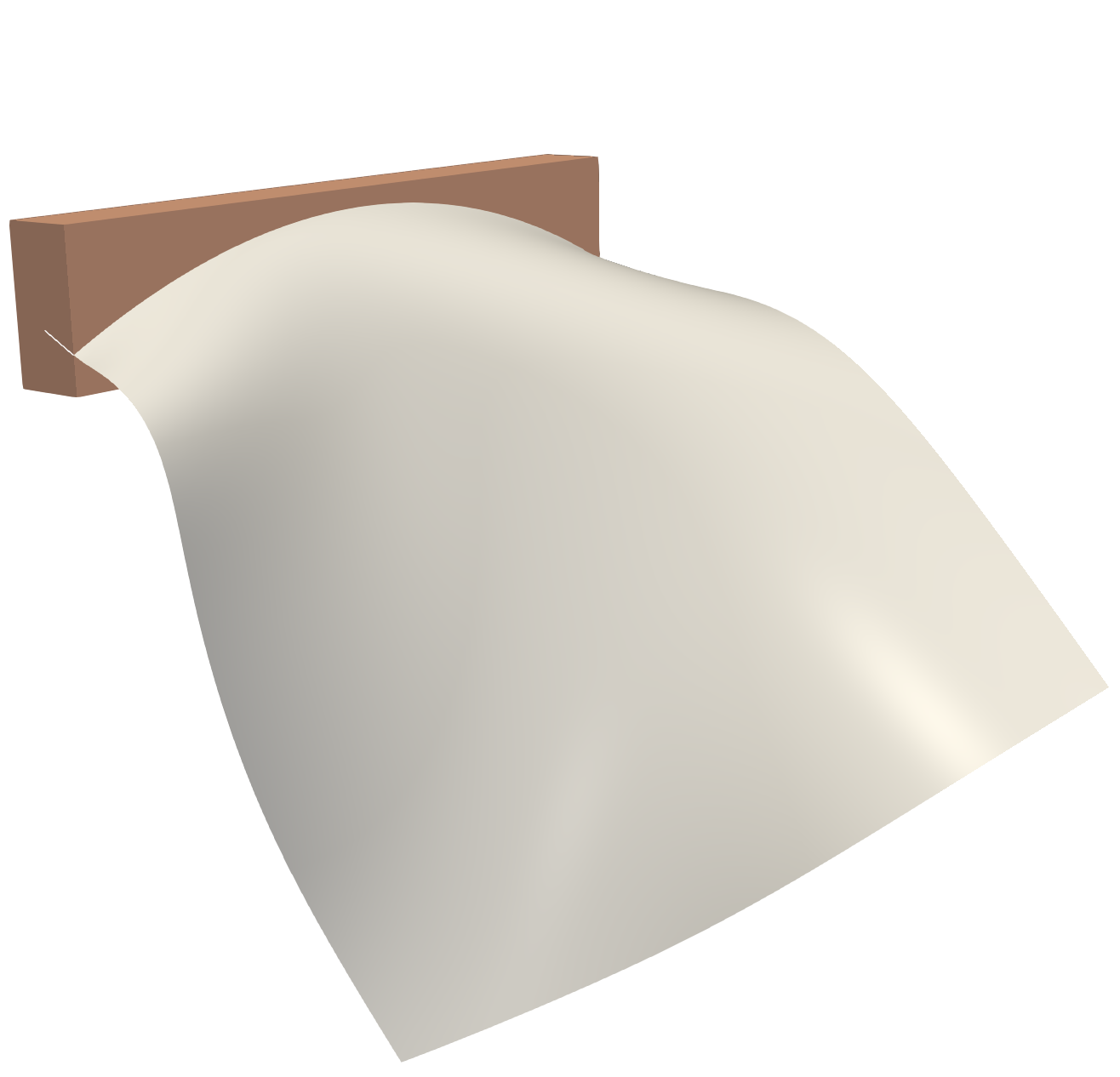}};
		\node at (7.4,1.8) {$\mathcal{M}_B$};
		\node[inner sep=0pt] (Undeformed) at (0,1.2) {\includegraphics[width=0.37\textwidth]{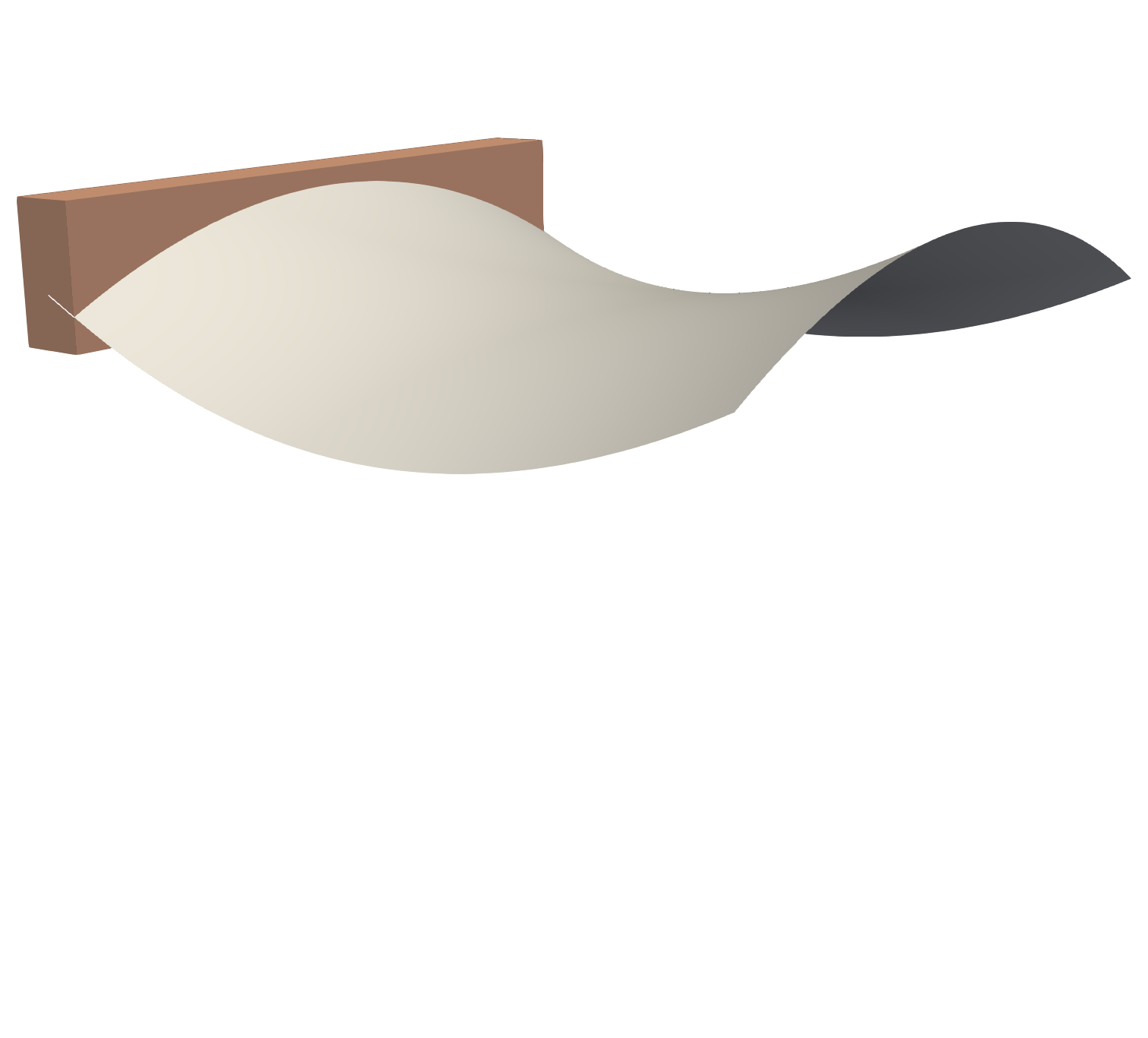}};
		\node at (-1.3,2.2) {$\mathcal{M}_A$};

		\node[inner sep=0pt] (Chart) at (3,0) {\includegraphics[width=0.25\textwidth]{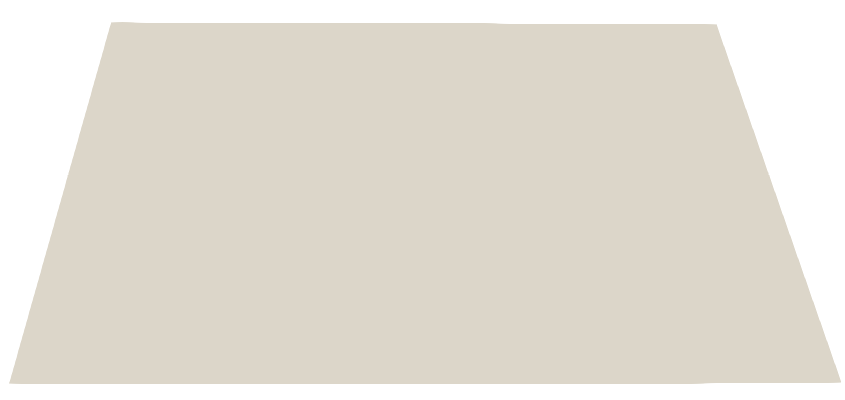}};
		\node at (3,-0.3) {$\omega \subset \R^2$};

		\path[->] (1.75, 0.) edge [bend left] node[left,xshift=-5] {$\psi_{A}$} (-0.6, 1.5);
		\path[->] (4.25, 0.) edge [bend right] node[right,xshift=5] {$\psi_{B}$} (6.6, 1.1);
		\path[->] (0.0, 2.1) edge [bend left] node[above, xshift=0] {$\phi = \psi_{B} \circ \psi_{A}^{-1}$} (6, 2.1);
	\end{tikzpicture} }
	\caption{A deformation $\phi$ of a parametrized surface $\mathcal{M}_A$ onto an image surface $\mathcal{M}_B$ with clamped boundary
	is described by parametrizations $\psi_A,\psi_B$ over the chart domain $\omega$.}
	\label{fig:overview}
\end{figure}
Note that the unit normal of $\mathcal{M}_B$ is defined for every $x\in \omega$ as
\begin{align*}
 n_B(x)  =  n[\psi_B](x) = \frac{\partial_{1} \psi_B(x) \times \partial_{2} \psi_B(x)}{\vert\partial_{1} \psi_B(x) \times \partial_{2} \psi_B(x)\vert} 
\end{align*}
and the corresponding first fundamental form $g_B$ and second fundamental form $\secFF_B$ at $x\in \omega$ are given by
\begin{align*}
 g_B(x) &= (\nabla \psi_B(x))^\top  \nabla \psi_B(x) \, , \\
 \secFF_B(x) &= \nabla n_B(x) \cdot \nabla \psi_B(x) =-D^2\psi_B(x) \cdot n_B(x) \, .
\end{align*}
Furthermore, the matrix representation of the shape operator $S_B$ of $\mathcal{M}_B$ at $x\in \omega$ is given by
$S_B(x) = g_B(x)^{-1}\secFF_B(x)$.
Normal, metric, second fundamental form and shape operator are analogously defined for $\mathcal{M}_A$ and the
given parametrization $\psi_A$.
To compare the shape operator $S_A$ of the undeformed configuration with the
shape operator $S_B$ of the deformed configuration, one considers
the matrix representation of the relative shape operator
\begin{align} \label{eq:relativeShapeOp}
  S^{\text{rel}}_{\psi_B} \coloneqq g_A^{-1}(\secFF_A - \secFF_B) = g_A^{-1} \left( D^2\psi_B \cdot n_B - D^2\psi_A \cdot n_A \right)\,.
\end{align}
Following Friesecke~\etal~\cite{FrJaMo03}, the bending energy $E_{\text{bend}} \colon H^2(\omega; \R^3) \to [0,\infty]$ 
of the deformed object is given by
\begin{align}\label{eq:bendingEnergy}
   E_{\text{bend}}[\psi_B] \coloneqq
   \begin{cases}
     \frac{\alpha}{2} \int_{\omega}{\sqrt{\det g_A}\ \tr\left(S^{\text{rel}}_{\psi_B}S^{\text{rel}}_{\psi_B}\right)} \d x & \text{if } g_B = g_A \, ,\\
     \infty & \text{otherwise} \, .
   \end{cases}
\end{align}
Note that $\tr\left(S^{\text{rel}}_{\psi_B}S^{\text{rel}}_{\psi_B}\right) = \left| g_A^{-\frac12}(\secFF_B - \secFF_A) g_A^{-\frac12} \right|^2$,
as it was derived for a prestrained plate model in \cite{BhLeSc16}.
A bending energy on prestrained plates involving the same integrand was also applied in \cite{BoGuNo21}.
Here, $\alpha>0$ denotes an elastic constant and the condition $g_B = g_A$ encodes the metric constraint on the map $\phi=\psi_B \circ \psi_A^{-1}$,
\ie
\begin{equation} \label{eq:isometryConstraint}
 (\nabla \psi_B(x))^\top  \nabla \psi_B(x) = g_B(x) = g_A(x) = (\nabla \psi_A(x))^\top \nabla \psi_A(x) \quad \text{for a. e. } x \in \omega \, .
\end{equation}
Moreover, we define the potential energy by
\begin{align*}
 E_{\text{pot}}[\psi_B] \coloneqq - \int_{\omega}{\sqrt{\det g_A}\ f \cdot \psi_B} \d x
\end{align*}
and consider clamped boundary conditions on $\Gamma_D \subset \partial \omega$ with $\mathcal{H}^{1}(\Gamma_D)>0$, \ie
\begin{align} \label{eq:clampedBdry}
  \psi_B = \psi_A\, , \quad \nabla \psi_B = \nabla \psi_A \quad \text{ on } \Gamma_D.
\end{align}
Finally, we ask for a minimizer $\psi_B$ of the total free energy
\begin{align}\label{eq:totalEnergy}
 E[\psi] \coloneqq E_{\text{bend}}[\psi] + E_{\text{pot}}[\psi]
 = \frac{\alpha}{2} \int_{\omega}{\sqrt{\det g_A}\, \tr\left(S^{\text{rel}}_{\psi}S^{\text{rel}}_{\psi}\right)} \d x - \int_{\omega}{\sqrt{\det g_A}\ f \cdot \psi} \d x
\end{align}
over all $\psi \in H^2(\omega; \R^3)$ satisfying the metric constraint~\eqref{eq:isometryConstraint} and the clamped boundary conditions \eqref{eq:clampedBdry}.
In a more general setting, other Dirichlet boundary conditions could be considered, such as $\psi_B = \varphi,\;\nabla \psi_B = \Phi$ on $\Gamma_D$. In general it is not clear if such a deformation $\psi_B$ satisfying  the metric constraint \eqref{eq:isometryConstraint} exists, even if $\Phi^\top\Phi = \nabla\psi_A^\top\nabla \psi_A$ on $\Gamma_D$ is satisfied. However, with the assumption that admissible deformations exist, the analysis presented in this paper could be extended to more general Dirichlet boundary conditions, as in \cite{Ba11}.

In this paper, we are primarily interested in a numerical approximation scheme for the above variational problem.
First, note that a conforming finite element approximation of the bending energy~\eqref{eq:bendingEnergy} would require globally $C^1$-elements, which are computationally demanding. 
As an alternative, Bonito \etal \cite{BoNoNt20} proposed a discontinuous Galerkin approach for isometric deformations of thin elastic plates and in \cite{BoGuNo21} Bonito \etal established a $\Gamma$-convergence theory of a local discontinuous Galerkin approach for prestrained plates.
Here, we follow Bartels~\cite{Ba11}, who made use of the discrete Kirchhoff Triangle (DKT) element
to approximate bending isometries in the case of deformations of thin elastic plates.
This approach has also been applied in \cite{Ba17} to approximate deformations of plates for a F\"{o}ppl--von K\'{a}rm\'{a}n model, which has been used to verify a break of symmetry for deformations of smooth, circular cones.
The key ingredient of the DKT element is a non-conforming second derivative with suitable approximation properties. Nodal-wise degrees of freedom for the Jacobian of the deformation
enable to restrict the isometry constraint to nodes of the underlying triangular mesh.
Furthermore, Bartels considered a linearization of the isometry constraint and a discrete gradient flow approach to minimize the energy.
In~\cite{HoRuSi20}, Hornung \etal applied the DKT element for a material optimization problem on thin elastic plates, where the isometry constraint was strictly enforced in a second order scheme.

Our goal is to extend the approximation result of Bartels to the case of curved surfaces $\mathcal{M}_A$
under the assumption that $\mathcal{M}_A$ is a parametrized surface as described above.
For isometric deformations in the flat case, the relative shape operator is symmetric and
the Frobenius norm of the relative shape operator is equal to the Frobenius norm of the second derivative of the deformation,
\ie
\begin{align} \label{eq:plateIsometrySimplification}
 \tr\left(S^{\text{rel}}_{\psi_B}S^{\text{rel}}_{\psi_B}\right)=\tr\left((S^{\text{rel}}_{\psi_B})^\top S^{\text{rel}}_{\psi_B}\right)=|S^{\text{rel}}_{\psi_B}|^2 = |D^2 \psi_B|^2 \, ,
\end{align}
which drastically simplifies the computational effort, since the second variation of the corresponding bending energy becomes 
independent of $\psi_B$.
In that respect, the central insight is a simplification of the relative shape operator similar to \eqref{eq:plateIsometrySimplification}.

The outline of this paper is as follows.
In Section~\ref{sec:reformulation}, we will rewrite the total elastic energy via a simplification of the relative shape operator and prove existence of a minimizing deformation.
In Section~\ref{sec:discrete}, the non-conforming finite element approximation via the Discrete Kirchhoff Triangle will be revisited and used to discretize the total elastic energy.
Instead of a linearization, we incorporate an exact metric constraint at nodal positions.
To solve the resulting constraint optimization problem numerically,
we take into account a Newton method for an associated Lagrangian in Section~\ref{sec:Newton}.
Finally, in Section~\ref{sec:results}, we discuss several selected examples and study the convergence behaviour experimentally.

\section{Reformulation of the bending energy}\label{sec:reformulation}
In this section, we will show an identity for the Frobenius norm of the relative shape operator $S^{\text{rel}}_{\psi_B}$ under the metric constraint.
This reformulation will ensure that the dependence of the elastic energy on second order derivatives of the parametrization $\psi_B$ is quadratic
and the remaining nonlinearity is a quadratic term involving the normal $n_B$ in the deformed configuration.

\begin{proposition}[transformed bending energy density]\label{prop:transform}
Let $\psi_B=(\psi_B^m)_{m=1,2,3} \in H^2(\omega; \R^3)$ with
$(\nabla \psi_B)^\top \nabla \psi_B= (\nabla \psi_A)^\top \nabla \psi_A$ almost everywhere in $\omega$.
Then, we have the identity 
	$\tr\left(S^{\text{rel}}_{\psi_B}S^{\text{rel}}_{\psi_B}\right)= |g_A^{-\frac12}(\secFF_B-\secFF_A)g_A^{-\frac12}|^2 = \transRelShapeOpt{\psi_B}$, where
\begin{align} \label{eq:transformedRelShapeOp}
  \transRelShapeOpt{\psi_B} \coloneqq 
  \sum_{m=1}^{3}|g_A^{-\frac12}D^2\psi_B^mg_A^{-\frac12}|^2 -2 \secFF_B :(g_A^{-1} \secFF_Ag_A^{-1}) + C_A
    \end{align}
	where the constant $C_A$ depends only on derivatives of $\psi_A$.
\end{proposition}
\begin{proof}
Differentiation of $\partb{i}\cdot \partb{i} = \parta{i} \cdot \parta{i}$ for $i \in \{1,2\}$ in direction $j \in \{1,2\}$ yields
$\partial_j \partb{i}\cdot \partb{i}  =  \partial_j \parta{i}\cdot \parta{i}$.
Similarly, differentiation of $\partb{i}\cdot \partb{j}=\parta{i}\cdot \parta{j}$ in direction $i \in \{1,2\}$ gives
$$ \partial_i^2 \psi_B \cdot \partb{j} + \partial_j \partb{i}\cdot \partb{i} = \partial_i^2 \psi_A \cdot \parta{j} + \partial_j \parta{i}\cdot \parta{i}$$
and taking into account the first identity we obtain $ \partial_i^2 \psi_B \cdot \partb{j} = \partial_i^2 \psi_A\cdot \parta{j}$.
Altogether, using that the parameter domain is two dimensional, we obtain
\begin{equation}
\partial_i \partial_j \psi_B \cdot \partial_k \psi_B = \partial_i \partial_j \psi_A \cdot \partial_k \psi_A \ \ \forall i,j,k\in \{1,2\}.
\label{eq:secondtimesfirstfinal}
\end{equation}
Next, we consider the Gram-Schmidt orthonormalization of the columns of the Jacobian $\nabla \psi_B$ 
\begin{align*}
 y_B^1 \coloneqq \frac{1}{\vert\partb{1}\vert}\partb{1}\,,\quad y_B^2 \coloneqq \frac{1}{\vert\hat{y}_B^2\vert}\hat{y}_B^2 \quad \text{ with  }\quad
 \hat{y}_B^2 \coloneqq \partb{2} - (\partb{2} \cdot y_B^1)y_B^1
 \end{align*}
and define $y_A^1,y_A^2,\hat{y}_A^2$ analogously for the parametrization $\psi_A$.\\
Then,  both $\{ y_B^1, y_B^2, n_B \}$ and $\{ y_A^1, y_A^2, n_A \}$ form an orthonormal basis of $\R^3$.
In particular, we get the orthogonal decomposition
\begin{equation}
	\partial_k \partb{j} = (\partial_k \partial_j \psi_B \cdot n_B)n_B + (\partial_k \partial_j \psi_B \cdot y_B^1)y_B^1 +(\partial_k \partial_j \psi_B \cdot y_B^2)y_B^2\,.
	\label{eq:newd2psi}
\end{equation}
By the metric constraint $\vert\parta{1}\vert = \vert\partb{1}\vert$ and $\parta{2} \cdot \parta{1} = \partb{2} \cdot \partb{1}$
and consequently
\begin{align*}
	y_B^1 &= \frac{1}{\vert\parta{1}\vert}\partb{1}\, ,\quad \hat{y}_B^2 = \partb{2} - (\parta{2}\cdot y_A^1)\frac{1}{\vert\parta{1}\vert}\partb{1}.
\end{align*}
Furthermore, we obtain
\begin{align*}
	\vert\hat{y}_B^2\vert^2 =& \vert\partb{2}\vert^2 - 2(\parta{2}\cdot y_A^1)\frac{1}{\vert\parta{1}\vert}\partb{1} \cdot \partb{2} + \left\vert(\parta{2}\cdot y_A^1)\frac{1}{\vert\parta{1}\vert}\right\vert^2 \vert\partb{1}\vert^2 \\
	=&  \vert\parta{2}\vert^2 - 2(\parta{2}\cdot y_A^1)\frac{1}{\vert\parta{1}\vert}\parta{1} \cdot \parta{2} + \left\vert(\parta{2}\cdot y_A^1)\frac{1}{\vert\parta{1}\vert}\right\vert^2 \vert\parta{1}\vert^2\\ =& \vert\hat{y}_A^2\vert^2.
\end{align*}
Taking into account \eqref{eq:secondtimesfirstfinal}, it follows that
\begin{align*}
\partial_k \partb{j} \cdot y_B^1 &=\frac{1}{\vert\parta{1}\vert}\partial_k \partb{j} \cdot \partb{1}=\frac{1}{\vert\parta{1}\vert}\partial_k \parta{j} \cdot \parta{1} =  \partial_k \parta{j} \cdot y_A^1\,,\\
	\partial_k \partb{j} \cdot y_B^2 &= \frac{1}{\vert\hat{y}_B^2\vert} \partial_k \partb{j} \cdot \hat{y}_B^2=\frac{1}{\vert\hat{y}_A^2\vert} \partial_k \partb{j} \cdot (\partb{2} - (\partb{2} \cdot y_B^1)y_B^1)\\
	&=\frac{1}{\vert\hat{y}_A^2\vert} \partial_k \parta{j} \cdot (\parta{2} - (\parta{2} \cdot y_A^1)y_A^1) = \partial_k \parta{j} \cdot y_A^2\,.
\end{align*}
Thus, we obtain
\begin{equation*}
	\partial_k \partb{j} = (\partial_k \partial_j \psi_B \cdot n_B)n_B + (\partial_k \partial_j \psi_A \cdot y_A^1)y_B^1 +(\partial_k \partial_j \psi_A \cdot y_A^2)y_B^2.
\end{equation*}
Next, we consider the integrand of the bending energy. 
Similar to the calculations in \cite{BaBoNo17} in the context of prestrained plates, we can write
	\begin{align*}
		\left| g_A^{-\frac12}(\secFF_B-\secFF_A)g_A^{-\frac12} \right|^2 = \left| g_A^{-\frac12}\secFF_B g_A^{-\frac12} \right|^2 - 2 (g_A^{-\frac12}\secFF_B g_A^{-\frac12} ):(g_A^{-\frac12}\secFF_A g_A^{-\frac12} ) + \left| g_A^{-\frac12}\secFF_A g_A^{-\frac12} \right|^2\,.
	\end{align*}
	The last term only depends on the initial configuration $\psi_A$ and can hence be neglected.
Let $g_A^{-\frac12} = \left(g^{-\frac12}_{A,ij}\right)_{i,j=1,2}$. 
 Using $\vert n_B\vert^2=1$ and the decomposition \eqref{eq:newd2psi}, $\left| g_A^{-\frac12}\secFF_B g_A^{-\frac12} \right|^2$ can be written as
\begin{align*}
&\sum_{i,j=1}^2\left| \sum_{k,l=1}^2g^{-\frac12}_{A,ik}g^{-\frac12}_{A,lj} \left(\partial_k \partial_l \psi_B\cdot n_B\right) \right|^2 = \sum_{i,j=1}^2\left| \sum_{k,l=1}^2g^{-\frac12}_{A,ik}g^{-\frac12}_{A,lj} \left(\partial_k \partial_l \psi_B\cdot n_B\right)n_B \right|^2 \\ =& \sum_{i,j=1}^2\left| \sum_{k,l=1}^2g^{-\frac12}_{A,ik}g^{-\frac12}_{A,lj} \left(\partial_k \partial_l \psi_B-\left[  \left(\partial_k \partial_l \psi_A\cdot y_A^1\right)y_B^1 + \left(\partial_k \partial_l \psi_A\cdot y_A^2\right)y_B^2\right]\right)\right|^2 \\
=& \sum_{i,j=1}^2\left| \sum_{k,l=1}^2g^{-\frac12}_{A,ik}g^{-\frac12}_{A,lj}\partial_k \partial_l \psi_B\right|^2 + \sum_{i,j=1}^2\left| \sum_{k,l=1}^2g^{-\frac12}_{A,ik}g^{-\frac12}_{A,lj}\left[  \left(\partial_k \partial_l \psi_A\cdot y_A^1\right)y_B^1 + \left(\partial_k \partial_l \psi_A\cdot y_A^2\right)y_B^2\right]\right|^2
\\ &- 2 \sum_{i,j=1}^2\!\left( \sum_{k,l=1}^2g^{-\frac12}_{A,ik}g^{-\frac12}_{A,lj} \partial_k \partial_l \psi_B\!\right)\!\cdot\! \left( \sum_{k,l=1}^2g^{-\frac12}_{A,ik}g^{-\frac12}_{A,lj} \left[ \left(\partial_k \partial_l \psi_A\cdot y_A^1\right)y_B^1 \!+\! \left(\partial_k \partial_l \psi_A\cdot y_A^2\right)y_B^2\right]\!\right).
\end{align*}
Since $|y_B^1|^2 = |y_B^2|^2=1$ and $y_B^1 \cdot y_B^2 = 0$, the second term on the right hand side is only depending on $\psi_A$ and can hence be regarded as constant. The same applies for the third term, considering the metric constraint and the calculations made above. Furthermore we can rewrite, using the symmetry of $g_A^{-\frac12}$
\begin{align*}
	&(g_A^{-\frac12}\secFF_B g_A^{-\frac12} ):(g_A^{-\frac12}\secFF_A g_A^{-\frac12} ) 
	= \sum_{i,j=1}^{2} \left(\sum_{k,l=1}^{2} g^{-\frac12}_{A,ik}\secFF_B^{kl}g^{-\frac12}_{A,lj}\right)\left(\sum_{m,n=1}^{2} g^{-\frac12}_{A,im}\secFF_A^{mn}g^{-\frac12}_{A,nj}\right)\\
	=& \sum_{k,l=1}^{2} \secFF_B^{kl}\left( \sum_{m,n=1}^{2} \left( \sum_{i=1}^{2} g^{-\frac12}_{A,ki} g^{-\frac12}_{A,im}\right) \secFF_A^{mn} \left( \sum_{j=1}^{2} g^{-\frac12}_{A,lj} g^{-\frac12}_{A,jn} \right) \right) = \secFF_B : (g_A^{-1}\secFF_Ag_A^{-1})
\end{align*}
 which proves the claim.
\end{proof}
As an immediate consequence, we obtain the following transformed total free energy.
\begin{corollary}[transformation of the total free energy]
Let $\psi_B \in H^2(\omega; \R^3)$ with
$(\nabla \psi_B)^\top \nabla \psi_B= (\nabla \psi_A)^\top \nabla \psi_A$ almost everywhere in $\omega$.
Then, the total free energy~\eqref{eq:totalEnergy} can up to a constant be rewritten as
\begin{align} \label{eq:transTotalEnergy}
 E[\psi_B] 
 = \frac{\alpha}{2} &\int_{\omega}\sqrt{\det g_A}\left( \sum_{m=1}^{3}|g_A^{-\frac12}D^2\psi_B^mg_A^{-\frac12}|^2 -2\secFF_B : (g_A^{-1}\secFF_Ag_A^{-1})\right) \d x \\
   - &\int_{\omega}{\sqrt{\det g_A}\ f \cdot \psi_B} \d x \nonumber
\end{align}
\end{corollary}
Using this reformulation, we obtain the following existence result.
\begin{theorem}[existence]\label{thm:existence}
For the set
\begin{equation*}
\mathcal{A} = \{ \psi\in H^2(\omega; \R^3)\; \vert \; (\nabla \psi)^\top \nabla \psi = (\nabla \psi_A)^\top \nabla \psi_A \text{ a.e. in }\omega;\ \psi=\psi_A,\nabla \psi= \nabla \psi_A \text{ on }\Gamma_D \}\end{equation*}
of admissible parametrizations subject to the metric constraint and clamped boundary conditions and for $f\in L^2(\omega, \R^3)$
there exists a parametrization $\psi_B \in \mathcal{A}$
which minimizes the total free energy $E[\psi]$ given in \eqref{eq:transTotalEnergy} over all $\psi\in \mathcal{A}$\,.
\end{theorem}
\begin{proof}
We begin remarking that $\mathcal{A}$ is nonempty, because $\psi_A\in\mathcal{A}$.
To apply the direct method in the calculus of variations we at first show the
uniform boundedness of a minimizing sequence in $H^2(\omega, \R^3)$.
To this end,
	we first estimate for $\psi \in \mathcal{A}$
\begin{align*}
E[\psi] \geq& 
\frac{\alpha}{2} \int_{\omega} \sqrt{\det g_A}\sum_{m=1}^{3}\left|g_A^{-\frac12}D^2\psi^mg_A^{-\frac12}\right|^2\d x \\
&- \alpha \int_{\omega} \sqrt{\det g_A}\left|D^2\psi \cdot n[\psi]\right| \left|g_A^{-1} \secFF_Ag_A^{-1}\right|\text{d}x 
- \int_{\omega} \sqrt{\det g_A}\ |f| |\psi| \d x\,.
\end{align*}
using
Cauchy-Schwartz' inequality. Note that $\left|D^2\psi \cdot n[\psi]\right| \leq \left|D^2\psi\right|$ again by Cauchy-Schwartz with $|n[\psi]| = 1$. 
Thus, since 
$g_A$ and $g_A^{-\frac12}$ are uniformly bounded, we obtain
\begin{align*}
E[\psi] \geq c \norm{D^2 \psi}_{L^2(\omega)}^2 - C\left( \intomega{\left|D^2\psi\right| \left|g_A^{-1} \secFF_Ag_A^{-1}\right|} + \intomega{|f||\psi|} \right)
\end{align*} for generic constants $c,\,C>0$ depending only on $\psi_A$ and $\alpha$.
Making use of  Poincare's inequality and Young's inequality,
we obtain
$E[\psi] \geq c \norm{D^2 \psi}_{L^2(\omega)}^2 - C$.
Next, let $(\psi_l)_l\subset \mathcal{A}$ be a minimizing sequence with
$\inf_{\psi\in \mathcal{A}} E[\psi] = \lim_{l\to\infty}E[\psi_l]$.
Then, the last estimate ensures that $\norm{\psi_l}_{H^2(\omega,\R^3)}\leq C$.
By the reflexivity of $H^2$, there exists a subsequence and a function $\psi_B\in H^2(\omega,\R^3)$ s.t. after a reindexing $\psi_l$ converges weakly to some 
$\psi_B$ in $H^2(\omega,\R^3)$.
By the Rellich--Kondrachov compactness theorem, we can extract another subsequence ensuring that $\nabla \psi_l(x)\to\nabla \psi_B(x)$ point-wise almost everywhere. 
Thus the limit $\psi_B$ also fulfills the metric constraint and by the trace theorem the clamped boundary conditions.
Hence, $\psi_B\in\mathcal{A}$.

Furthermore, the sequence of normal fields $n[\psi_l]$ is uniformly bounded in  $L^\infty(\omega, \R^3)$ and $n[\psi_l]$ converges
point-wise almost everywhere to $n[\psi_B]$. Altogether,
 $D^2\psi_l \cdot n[\psi_l]$
 converges weakly in $L^2(\omega,\R^{2\times2})$ to
$D^2\psi_B \cdot n[\psi_B] =  \secFF_B$.
Finally, the convexity of $|\cdot|^2$ implies that the total free energy $E[\cdot]$ is weakly lower semi-continuous and thereby
\begin{align*}
E[\psi_B] \leq \liminf\limits_{l\to\infty}E[\psi_l] = \inf_{\psi \in \mathcal{A}} E[\psi] \,.
\end{align*}
\end{proof}
\section{Discretization based on the Discrete Kirchhoff Triangle} \label{sec:discrete}
In this section, we will derive a non-conforming finite element discretization of the total free energy and the corresponding discrete
metric constraint. 
This derivation follows the general approach proposed by Bartels for discrete deformations of plates in~\cite{Ba11}.
In addition, we refer to the monograph~\cite{Ba15}.
At first, let us review the non-conforming finite element approximation based on the Discrete Kirchhoff Triangle (DKT).
For simplicity, we directly assume that $\omega$ is a polygonal parameter domain.
Let $\mathcal{T}_h$ be a regular triangulation of $\omega$ with maximal triangle diameter $h>0$.
We denote by $\mathcal{N}_h$ the set of vertices and by $\mathcal{E}_h$ the set of edges.
For $k\in \mathbb{N}$, we denote by $\mathcal{P}_k$ the set of polynomials of degree at most $k$.
For vertices $z_1,z_2,z_3\in \mathcal{N}_h$ of a triangle $T$ we define $z_T = (z_1+z_2+z_3)/3$ as the center of mass of $T$ 
and introduce the reduced space of cubic polynomials
\begin{align*}
\mathcal{P}_{3,\text{red}}(T) \coloneqq \left\{ p\in \mathcal{P}_3(T) \; \Big\vert \; 6p(z_T)
= \sum_{i = 1,2,3}\left( 2p(z_i) - \nabla p(z_i)\cdot (z_i - z_T) \right) \right\}
\end{align*}
which still has $\mathcal{P}_{2}$ as a subspace and the finite element spaces
\begin{align*}
	{\bf{W}}_h \coloneqq& \left\{ w_h \in C(\widebar{\omega}) \; \vert \; w_h \restr{T} \in \mathcal{P}_{3,\text{red}}(T) \text{ for all }T \in \mathcal{T}_h  \text{ and }\nabla w_h \text{ is continuous at }\mathcal{N}_h \right\}\,, \\
	{\bf{\Theta}}_h \coloneqq& \left\{ \theta_h \in C(\widebar{\omega};\R^2) \; \vert \; \theta_h\restr{T} \in \mathcal{P}_2(T)^2 \text{ and }\theta_h \cdot n_E\restr{E} \text{ is affine for all } E\in \mathcal{E}_h  \right\}\, .
\end{align*}
For a function $w \in H^3(\omega)$, the interpolation $w_h = \mathcal{I}^{DKT} w \in {\bf{W}}_h$ is defined on every triangle $T\in \mathcal{T}_h$ by 
$w_h(z) = w(z)$ and $\nabla w_h(z) = \nabla w(z)$ for all vertices $z\in \mathcal{N}_h \cap T$, 
which is well-defined due to the continuous embedding of $H^3(\omega)$ into $C^1(\widebar \omega)$.
The discrete gradient operator $\theta_h \colon {\bf{W}}_h \to {\bf{\Theta}}_h$ is defined via
\begin{align*}
\theta_h[w_h](z) = \nabla w_h(z) \, ,  \quad \theta_h[ w_h](z_E)\cdot t_E = \nabla w_h(z_E) \cdot t_E
\end{align*}
for all vertices $z \in \mathcal{N}_h$, all edges $E\in \mathcal{E}_h$ with $t_E$ denoting a unit tangent vector on $E$, and $z_E$ the midpoint of $E$.
We use superscripts $(\theta^j_h[w_h])_{j=1,2}$
to indicate the components of $\theta_h[w_h]$ corresponding to an approximation of $\partial_j w_h$.
The operator $\theta_h$ can analogously be defined on $H^3(\omega)$.
This operator has the following properties (\cf~Bartels~\cite{Ba11} and the textbook by Braess~\cite{Br13} for the proofs):

There exists constants $c_0,\,c_1,\,c_2, \,c_3 > 0$ such that for $T \in \mathcal{T}_h$ with $h_T=\mathrm{diam(T)}$, 
$w\in H^3(T)$ and $w_h \in {\bf{W}}_h$  
\begin{align} 
\label{eq:interpolDKT}
&\Vert w-\mathcal{I}^{DKT} w \Vert_{H^m(T)}  \leq c_0 h_T^{3-m} \Vert w\Vert_{H^3(T)} \quad \text{for } m = 0,1,2,3 \,,\\
\label{eq:ia}
& c_1^{-1} \lznormT{D^{k+1}w_h} \leq \lznormT{D^k\theta_h[ w_h]} \leq c_1 \lznormT{D^{k+1}w_h} \quad \text{for } k = 0,1\,, \\
\label{eq:ib}
&\lznormT{\theta_h[ w_h] - \nabla w_h} \leq c_2 h_T \lznormT{D^2 w_h}\,,\\
\label{eq:ic}
&\lznormT{\theta_h[ w] - \nabla w} + h_T\lznormT{\nabla \theta_h[ w] - D^2 w} \leq c_3 h_T^2 \norm{w}_{H^3(T)}\,.
\end{align}
Furthermore, we have the following norm property:
The mapping $w_h \mapsto \lznormO{\nabla \theta_h[ w_h]}$ defines a norm on
$\left\{ w_h \in {\bf{W}}_h  \; \vert \; \ w_h(z)=0,\ \nabla w_h(z) = 0 \text{ for all }z \in \mathcal{N}_h \cap \Gamma_D \right\} $.

Based on the DKT element we are now able to discretize the total free energy \eqref{eq:transTotalEnergy}.
More precisely, we consider $\psi_h \in {\bf{W}}_h^3$ to approximate a parametrization $\psi \in H^2(\omega;\R^3)$.
Then, a discrete unit normal field $n[\psi_h]$ is defined as
$n[\psi_h] \coloneqq \frac{1}{|\partial_{1}\psi_h\times \partial_{2}\psi_h|}\partial_{1}\psi_h\times \partial_{2}\psi_h$,
and we apply $\nabla \theta_h[\psi_h]$ as a discrete (non-conforming) approximation of the Hessian $D^2 \psi_h$ in $L^2(\omega; \R^{3,2,2})$.
Furthermore, for the given fixed parametrization $\psi_A\in H^3(\omega; \R^3)$ define 
$\psi_{A,h} = \mathcal{I}^{DKT} \psi_A \in {\bf{W}}_h^3$.
Then, we define
$g_h = (\nabla \psi_{A,h})^\top\nabla \psi_{A,h}$,
$\secFF_h = \nabla \theta_h[\psi_{A,h}]\cdot n[\psi_{A,h}]$.
Let us assume that $\psi_A\in H^3(\omega; \R^3)$ and $f_h\in L^2(\omega; \R^3)$. 
In particular, the above coefficients depending on $\psi_A$ are well approximated
by their discrete counterparts obtained by the interpolation in ${\bf{W}}_h^3$ 
taking into account the estimates on the discrete gradient operator.
Note that instead of a DKT interpolation $\psi_{A,h}$, we could directly consider $\psi_A$ evaluated at quadrature points.
However, if one applies the presented approach to shape optimzation (cf.~\cite{BuClCo21}) one usually wants to optimize
$\psi_A$. Then, a DKT discretization $\psi_{A,h}$ would enable to actually perform such an optimization.
Now, the discrete transformed bending energy density is given by
\begin{align} \label{eq:transformedRelShapeOpDiscrete}
 \transRelShapeOpth{\psi_h} \coloneqq
  \sum_{m=1}^{3}|g_h^{-\frac12}\nabla \theta_h[\psi_h^m]g_h^{-\frac12}|^2 -2 (\nabla \theta_h[\psi_{h}]\cdot n[\psi_{h}]) :(g_h^{-1} \secFF_hg_h^{-1})
\end{align}
Correspondingly, the discrete total free energy is defined as
\begin{align} \label{eq:discreteEnergy}
 E_h[\psi_h]=
   \frac{\alpha}{2} \int_\omega \sqrt{\det g_h}  \transRelShapeOpth{\psi_h}  \d x
 - \int_\omega \sqrt{\det g_h}\ f_h \cdot \psi_h \d x \, .
\end{align}
We aim at minimizing this discrete energy over the following constraint set of discrete deformations
\begin{align*}
\mathcal{A}_h = \Big\{ \psi_h\in {\bf{W}}_h^3\, \Big\vert \, (\nabla \psi_h(z))^\top \nabla \psi_h(z) = (\nabla \psi_A(z))^\top \nabla \psi_A(z) \quad \forall z\in \mathcal{N}_h; \;&\\
\psi(z)=\psi_A(z), \nabla \psi(z) = \nabla \psi_A(z)\quad \forall z \in  \mathcal{N}_h \cap \Gamma_D &\Big\}.\nonumber
\end{align*}
Since $\psi_{A,h} \in \mathcal{A}_h$, this set is not empty.
In explicit, we require the metric constraint only on the nodes of the triangulation and the clamped boundary condition is applied solely on boundary nodes.

Now, we are in the position to formulate our main theorem on the approximation of large-scale isometric deformations of parametrized surfaces
minimizing the total free energy in case of clamped boundary conditions.
\begin{theorem}[convergence of discrete solutions]\label{theoapprox}
Let $(\mathcal{T}_h)_h$ be a sequence of uniformly regular triangulations of $\omega$ with maximal triangle diameter $h>0$.
Furthermore, let $\psi_A\in H^3(\omega; \R^3)$ and $f\in L^2(\omega; \R^3)$ and $(f_h)_h$ be a sequence of force fields in $L^2(\omega,\R^3)$ weakly converging to $f$ in $L^2(\omega,\R^3)$.
Assume that there exists a minimizer $\psi_B$ of the 
continuous total free energy 
$E[\cdot]$ \eqref{eq:transTotalEnergy} on $\mathcal{A}$ which can be approximated in $H^2(\omega; \R^3)$ by functions $\psi_\varepsilon \in H^3(\omega; \R^3) \cap \mathcal{A}$.
Then, for every $h \leq \bar h$, for $\bar h$ sufficiently small,
there exists a minimizer $\psi_h \in {\bf{W}}_h^3$ of the discrete total free energy $E_h[\cdot]$~\eqref{eq:discreteEnergy} on $\mathcal{A}_h$.
Furthermore, for $(E_h[\cdot])_h$ with $h \to 0$, let $(\psi_h)_h $ be a sequence of minimizers. Then
$$\norm{\theta_h[\psi_h]}_{H^1(\omega; \R^{3\times 2})} + \norm{ \psi_h}_{H^1(\omega; \R^{3})}\leq C$$
and there exists a subsequence which converges strongly in
$H^1(\omega;\R^3)$ to some\\ $\psi^\ast \in H^2(\omega;\R^3) \cap \mathcal{A}$.
Furthermore, $\psi^\ast$ is a minimizer of the energy $E[\cdot]$
defined in \eqref{eq:transTotalEnergy} on $\mathcal{A}$.
\end{theorem}
Finally, let us remark that in the flat case $g_A= I_2$ with $\psi_B\in H^2(\omega; \R^3)$ and $g_B = I_2$    
the mapping $\psi_B$ can be approximated in the strong $H^2$-topology by smooth isometries 
as shown by Hornung in \cite{Ho11}. 
We also refer to the monograph by Bartels~\cite{Ba15} for further properties of isometries in the flat case.
In the curved case, to the best of our knowledge, such a density result is unclear, since the proof in the flat case is based on the developability by Hartman and Nirenberg in \cite{HaNi59}.
\begin{proof}
The general procedure of this proof follows the basic procedure of the convergence proof given in \cite{Ba11} for the case of plates and in \cite{BaBoNo17} for the case of bilayer plates and in \cite{BoGuNo21} for prestrained plates and uses
$\Gamma$-convergence arguments.
With a slight misuse of notation we do not perform a reindexing when subsequences are selected.
Let $\mathcal{I}_h[\cdot]$ be the nodal interpolation operator mapping into the space of piece-wise affine, globally continuous functions in $\mathcal{T}_h$.

At first, using similar arguments as in the proof of Theorem~\ref{thm:existence} 
we can bound the discrete energy $E_h[\psi_h]$ for $\psi_h \in {\bf{W}}_h^3$ from below
\begin{align}\label{eq:Eest}
E_h[\psi_h] \geq c \norm{\nabla \theta_h[\psi_h]}_{L^2(\omega)}^2 
\!\!- C \!\left( \norm{\nabla \theta_h[\psi_h]}_{L^2(\omega)}\norm{n[\psi_h]}_{L^2(\omega)}
\!\!+ \! \norm{f}_{L^2(\omega)}\norm{\psi_h}_{L^2(\omega)}\right),
\end{align}
where $\norm{n[\psi_h]}_{L^2(\omega)}^2$ equals the area of  $\omega$.
Based on the nodal metric constraint, which implies $|\nabla \psi_h(z)|^2 = |\nabla \psi_A(z)|^2$ for all $z\in \mathcal{N}_h$, and applying an inverse inequality, see \cite{Br13}, we obtain for all $T\in \mathcal{T}_h$:
\begin{align}
& \norm{|\nabla \psi_h|^2 - \mathcal{I}_h[|\nabla \psi_A|^2]}_{L^1(T)} \leq C h_T^2 \norm{D^2(|\nabla \psi_h|^2)}_{L^1(T)}\label{eq:interpolestim}\\  &\leq C h_T^2 \left(\norm{D^3\psi_h}_{L^2(T)} \lznormT{\nabla \psi_h} + \lznormT{D^2\psi_h}^2\right) \leq C h_T\norm{D^2\psi_h}_{L^2(T)} \lznormT{\nabla \psi_h}.\nonumber
\end{align}
Now, using the  triangle inequality, Young's inequality, the nodal metric constraint,  and the norm equivalence estimates \eqref{eq:ia} we obtain
\begin{align*}
\lznormT{\nabla \psi_h}^2 \leq C h_T \left(\lznormT{\nabla \theta_h[\psi_h]}^2 + \lznormT{\theta_h[\psi_h]}^2\right) + \norm{\mathcal{I}_h[|\nabla \psi_A|^2]}_{L^1(T)} \, ,
\end{align*}
and with summation over all $T\in \mathcal{T}_h$ we get
\begin{align*}
 \lznormO{\nabla \psi_h}^2 \leq C h \left(\lznormO{\nabla \theta_h[\psi_h]}^2 + \lznormO{\theta_h[\psi_h]}^2\right) + \norm{\mathcal{I}_h[|\nabla \psi_A|^2]}_{L^1(\omega)}.
\end{align*}
Taking into account the clamped boundary conditions and applying Poincar\'e's inequality for $\theta_h[\psi_h]$ we  achieve
$
\lznormO{\nabla \psi_h}^2 \leq C h \lznormO{\nabla \theta_h[\psi_h]}^2 + C\,.
$
Now, applying Poincar\'e's inequality for $\psi_h$ and Young's inequality, we obtain
$\norm{\psi_h}_{L^2(\omega)} \leq C (1+h  \lznormO{\nabla \theta_h[\psi_h]})$.
Thus, using that $\lznormO{f_h}$ is uniformly bounded, and using again Young's inequality for the term $C\norm{\nabla \theta_h[\psi_h]}_{L^2(\omega)}\norm{n[\psi_h]}_{L^2(\omega)}$ we obtain
\begin{equation}\label{coercOfDiscEn}
E_h[\psi_h] \geq c \norm{\nabla \theta_h[\psi_h]}_{L^2(\omega)}^2 - C
\end{equation}
 for $h$ small enough. From this, the continuity of  $E_h[\cdot]$ on ${\bf{W}}^3_h$, and the 
norm property of  $\psi_h \mapsto \lznormO{\nabla \theta_h[\psi_h]}$ the existence of a minimizer $\psi_h$ of
$E_h[\cdot]$ follows for $h$ sufficiently small and $\lznormO{\nabla \theta_h[\psi_h]} \leq C$.
Then, Poincar\'e's inequality yields $\lznormO{\theta_h[\psi_h]} \leq C$.
Applying once more the norm equivalence estimates \eqref{eq:ia} we obtain $\lznormO{\nabla \psi_h} \leq C$\,.

Now, we consider the $\liminf$ inequality. 
By reflexivity of $H^1$,
there exist functions $\psi^\ast \in H^1(\omega; \R^3)$ and $\theta^\ast \in H^1(\omega; \R^{3\times 2})$, 
such that (up to subsequences) $\psi_h$ converges weakly to $\psi^\ast$ in $H^1(\omega; \R^3)$ and
$\theta_h[\psi_h]$ converges weakly to $\theta^\ast$ in $H^1(\omega; \R^{3\times 2})$\,.
Furthermore, one observes by \eqref{eq:ib}
$$\lznormO{\nabla \psi_h - \theta_h[\psi_h]} \leq c h \lznormO{\nabla \theta_h[\psi_h]} \leq Ch\,.$$
By the Rellich--Kondrachov theorem, $\theta_h[\psi_h]$ converges strongly to $\theta^\ast$ in $L^2(\omega;\R^3)$ for another subsequence.
Thus, the strong convergence $\nabla \psi_h$ to $\theta^\ast$ and the weak convergence of  $\nabla \psi_h$ to $\nabla \psi^\ast$ yields $\nabla \psi^\ast = \theta^\ast$
and in particular $\psi^\ast \in H^2(\omega; \R^3)$.
The continuity of the trace operator $H^2(\omega; \R^3) \to H^1(\Gamma_D; \R^{3})$ 
and interpolations estimates imply that $\psi^\ast$ fulfills the clamped boundary conditions.
To verify that $\psi^\ast$ fulfills the metric constraint we estimate
\begin{align*}
&\norm{(\nabla \psi_h)^\top \nabla \psi_h - (\nabla \psi_A)^\top \nabla \psi_A}_{L^1(T)}\\ &\leq
\norm{(\nabla \psi_h)^\top \nabla \psi_h - \mathcal{I}_h[(\nabla \psi_A)^\top \nabla \psi_A]}_{L^1(T)}
\!\!+ \!\norm{\mathcal{I}_h[(\nabla \psi_A)^\top \nabla \psi_A] - (\nabla \psi_A)^\top \nabla \psi_A}_{L^1(T)}\\
&\leq C h_T \lznormT{D^2 \psi_h}\lznormT{\nabla \psi_h}
+ C h_T^2\left( \lznormT{D^3 \psi_A}\lznormT{\nabla \psi_A} + \lznormT{D^2\psi_A}^2\right)\,.
\end{align*}
Here, we applied similar interpolation error estimates as in \eqref{eq:interpolestim}.
Summation over $T\in \mathcal{T}_h$ and the fact that $\nabla \psi_h \to \nabla \psi^\ast$ strongly in $L^2$ finally imply that
$(\nabla \psi^\ast)^\top \nabla \psi^\ast = (\nabla \psi_A)^\top \nabla \psi_A$ a.e. in $\omega$\,.
Since $\left(n[\psi_h]\right)_h$ is a bounded sequence in $L^2(\omega, \R^3)$
and $\nabla \psi_h$ converges point-wise to $\nabla \psi^\ast$ a.e.  $n[\psi_h] \to n[\psi^\ast]$ in $L^\infty(\omega, \R^3)$.
Furthermore, due to interpolation estimates, $g_h^{-\frac12}$, $g_h^{-1}$ and  $\secFF_h$
converge strongly to $g_A^{-\frac12}$, $g_A^{-1}$ and $\secFF_A$, respectively.
Altogether, recalling the definitions \eqref{eq:transformedRelShapeOp} and \eqref{eq:transformedRelShapeOpDiscrete},
we finally achieve the $\liminf$-inequality
\begin{align*}
E[\psi^\ast] \leq \liminf\limits_{h\to 0}E_h[\psi_h].
\end{align*}
With respect to the definition of a recovery sequence, we consider a function $\psi\in H^3(\omega; \R^3) \cap \mathcal{A}$. For $h>0$, let $\psi_{h}= \mathcal{I}^{DKT} \psi \in {\bf{W}}_h^3$ be the interpolation of $\psi$ defined on every triangle $T\in \mathcal{T}_h$ by $\psi_h(z) = \psi(z)$ and $\nabla \psi_h(z) = \nabla \psi(z)$ for all vertices $z\in \mathcal{N}_h \cap T$.
Taking into account \eqref{eq:interpolDKT}, \eqref{eq:ia}, and \eqref{eq:ic} 
we have for every $T\in \mathcal{T}_h$
\begin{align}\label{eq:interpol} 
\nonumber
&\lznormT{\theta_h[ \psi_{h}] - \nabla \psi} + h_T\lznormT{\nabla \theta_h[ \psi_{h}] - D^2 \psi} \\
\nonumber
&\leq \lznormT{\theta_h[ \psi_{h}-\psi]} + \lznormT{\theta_h[\psi]-\nabla \psi} \\
&\qquad + h_T \left( \lznormT{\nabla \theta_h[ \psi_{h}-\psi]}+   \lznormT{ \nabla \theta_h[\psi]- D^2 \psi}\right)  \leq c_3 h_T^2 \norm{\psi}_{H^3(T)}.
\end{align}
Using the estimate
$$\left\vert \frac{a}{\vert a \vert} - \frac{b}{\vert b\vert}\right\vert^2
= 2 \left(1- \frac{\vert b\vert}{\vert a\vert} + \frac{(b-a)\cdot b}{\vert a\vert\,\vert b\vert}\right)
\leq 2 \left( \frac{\vert a\vert-\vert b\vert}{\vert a\vert} + \frac{\vert b-a\vert}{\vert a\vert}\right)
\leq 4 \frac{\vert b-a\vert}{\vert a\vert}
$$
for $a=\partial_{1} \psi \times \partial_{2} \psi$,  $b =\partial_{1} \psi_h \times \partial_{2} \psi_h$,
and the identity $|\partial_{1} \psi \times \partial_{2} \psi| = \sqrt{\det((\nabla \psi)^\top \nabla \psi)}  = \sqrt{\det g_A}$
which follows from the metric constraint, 
we get
\begin{equation*}
 \intomega{ \sqrt{\det g_A} \ |n[\psi] - n[\psi_{h}]|^2 } \leq 4\intomega{|(\partial_{1}{\psi}_h\times \partial_{2}{\psi}_h)-(\partial_{1} \psi \times \partial_{2} \psi)|}.
\end{equation*}
Furthermore, by the interpolation estimate \eqref{eq:interpolDKT} we obtain
\begin{align*}
&\norm{\partial_{1}\psi_m \partial_{2}\psi_l - \partial_{1}{\psi_{h,m}}\partial_{2}{\psi_{h,l}}}_{L^1(\omega)}\\ 
&\leq \norm{\partial_{1}\psi_m - \partial_{1}{\psi_{h,m}}}_{L^2(\omega)}\norm{\partial_{2}\psi_{l}}_{L^2(\omega)} + \norm{\partial_{1}{\psi_{h,m}}}_{L^2(\omega)} \norm{\partial_{2}\psi_l - \partial_{2}{\psi_{h,l}}}_{L^2(\omega)} \leq C h^2 \norm{\psi}_{H^3(\omega)}^2
\end{align*}
and hence
$$\left(\intomega{\sqrt{\det g_A}\ \big|n[\psi] - n[\psi_h]\big|^2 }\right)^{\frac{1}{2}} \leq  C h \norm{\psi}_{H^3(\omega)}\,.$$
Now, let $\psi_B \in \mathcal{A}$ be a minimizing isometry for $E[\cdot]$. In~Theorem~\ref{thm:existence}, it is shown that such a minimizer exists.
By our assumption, we have
\begin{align*}
 \forall \varepsilon > 0 \ \exists \psi_\varepsilon \in H^3 (\omega;\R^3)\cap \mathcal{A}:\ \norm{\psi_B - \psi_\varepsilon}_{H^2(\omega;\R^3)} < \varepsilon\,.
\end{align*}
Applying the above estimates to $\psi_\varepsilon$ and its interpolation $\psi_{\varepsilon,h} = \mathcal{I}^{DKT} \psi_\varepsilon$ in
${\bf{ W}}_h^3$ and using the estimates~\eqref{eq:interpolDKT}, \eqref{eq:ic}  and \eqref{eq:interpol} we achieve
\begin{align} \label{eq:uniform}
&\intomega{\sqrt{\det g_h}\   \transRelShapeOpth{\psi_{\varepsilon,h}}  } 
\leq \intomega{\sqrt{\det g_A}\   \transRelShapeOpt{\psi_{\varepsilon}} } 
+ C h \norm{\psi_\varepsilon}_{H^3}\\ 
&\leq \intomega{\sqrt{\det g_A}\  \transRelShapeOpt{\psi_{B}}  } 
+ C \left(  \varepsilon +   h \norm{\psi_\varepsilon}_{H^3}\right)\nonumber
\end{align}
Now, we choose $h=h(\varepsilon)$ small enough such that
\begin{align}\label{eq:heps}
h(\varepsilon) \norm{\psi_\varepsilon}_{H^3(\omega)} < \varepsilon\,.
\end{align}
and use for the estimation of the potential energy that $\psi_{\epsilon,h(\epsilon)}$ converges strongly to $\psi_B$ in $L^2$
to obtain 
\begin{align*}
	 \limsup\limits_{\varepsilon\to 0} E_{h(\varepsilon)}[\psi_{\varepsilon,h(\varepsilon)}] \leq E[\psi_B].
\end{align*}
Finally, we get
\begin{align*}
		E[\psi^\ast ]\leq \liminf\limits_{\epsilon\to 0} E_{h(\varepsilon)}[\psi_{h(\varepsilon)}] \leq \limsup\limits_{\varepsilon\to 0} E_{h(\varepsilon)}[\psi_{\varepsilon,h(\varepsilon)}] \leq E[\psi_B]= \min_{\tilde \psi \in \mathcal{A}} E[\tilde \psi] \leq E[\psi^\ast].
\end{align*}
Hence, $\psi^\ast$ is a minimizer of $E[\cdot]$.
\end{proof}
In fact, the coupling of $h$ and the $H^3$-norm of the approximations determines the rate of convergence. 
This rate cannot be predicted under the assumption of this theorem.

\section{Implementation via Newton's method} \label{sec:Newton}
Now, we will describe the numerical implementation to minimize the discrete total energy $E_h$ as defined in~\eqref{eq:discreteEnergy}
over all discrete isometries $\phi_h \in \mathcal{A}_h$.
First, we observe that a function
$\psi_h \in \mathcal{A}_h = \{ \phi_h \in {\bf{W}}_h^3\; \vert \; \phi_h(z)=\psi_A(z),\nabla \phi_h(z)= \nabla \psi_A(z) \text{ on }\Gamma_D \}$
is determined by its values at the nodes and the values of the gradient at the nodes.
So, for the  discrete constraint minimization problem, there are $9 \times \vert \mathcal{N}_h \setminus \Gamma_D \vert$ degrees of freedom.
To implement the nodal-wise metric constraint, we define the Lagrangian
\begin{align*}
 L_h[\psi_h,p_h] \coloneqq E_h[\psi_h] - G_h[\psi_h](p_h)\,.
\end{align*}
Here, $E_h[\psi_h]$ is the discrete total free energy and
\begin{align*}
    G_h[\psi_h](p_h) \coloneqq
    \intomega{\mathcal{I}_h\left(
    \left[
    \left(\nabla \psi_h(z)\right)^\top \nabla \psi_h(z)
    - \left(\nabla \psi_A(z)\right)^\top \nabla \psi_A(z)
    \right] : p_h
    \right)}
\end{align*}
with Lagrange multiplier $p_h \in {\bf S}^{2,2}_h$, where ${\bf S}^{2,2}_h$ denotes the space of continuous piece-wise affine, symmetric matrices in $\R^{2,2}$.
In particular, $G_h[\psi_h](p_h)=0$ for all $p_h \in {\bf S}^{2,2}_h$ is equivalent to an enforcement of the metric constraint on all nodes of the triangulation.
The saddle point conditions are
\begin{align*}
\partial_{\psi_h} L_h[\psi_h,p_h](v_h) = 0\,,\quad
\partial_{p_h} L_h[\psi_h,p_h](q_h) = 0
\end{align*}
for all  $v_h \in \{\phi_h \in {\bf{W}}_h^3\; \vert \; \phi_h(z)=0,\nabla \phi_h(z)= 0 \text{ on }\Gamma_D \}$ and for all $q_h \in {\bf S}^{2,2}_h$\,.
To compute a saddle point, we use the IPOPT software library presented in~\cite{WaBi06}.
More precisely, we apply a Newton scheme for the Lagrangian which requires the computation of the first and second variations of the discrete energy $E_h[\cdot]$ and of
$G_h[\cdot](\cdot)$, respectively.
In IPOPT this corresponds to setting 
``hessian\_approximation'' to ``exact''.
Here, we take into account the default backtracking strategy by setting ``line\_search\_method'' to ``filter''. As stopping criterion we set ``tol'' to $10^{-12}$.
For the required integral evaluations, we implemented a Gauss quadrature of degree 6 with 12 quadrature points.
For the ease of presentation, we consider the continuous Lagrangian
$$
L[\psi,p] = E[\psi] - G[\psi](p)
$$
with $G[\psi](p)=\intomega{\left((\nabla \psi)^\top \nabla \psi - (\nabla \psi_A)^\top \nabla \psi_A\right) : p}$
and provide first and second variations of $E[\cdot]$ and of $G[\cdot](\cdot)$, respectively. Here, $X:Y$ denotes the canonical scalar product for tensors $X$ and $Y$. The transfer to the discrete counterparts is straightforward.
The energy is given by
$E[\psi] = \frac{\alpha}{2} \intomega{\sqrt{\det g_A}  \transRelShapeOpt{\psi}}-\intomega{\sqrt{\det g_A} f \cdot \psi}$, where we can write $\transRelShapeOpt{\psi} = \sum_{m=1}^{3}\left(g_A^{-1}D^2\psi^mg_A^{-1}\right):D^2\psi^m-2\left(D^2\psi \cdot n[\psi]\right) :\left(g_A^{-1} \secFF_Ag_A^{-1}\right)$.
For the first and second variation we obtain
\begin{align*}
\partial_\psi E[\psi](v)  &=\frac{\alpha}{2} \intomega{  \sqrt{\det g_A}  \partial_\psi \transRelShapeOpt{\psi}(v)} -\intomega{\sqrt{\det g_A} f \cdot v}\,, \\
\partial_\psi^2 E[\psi](v,w) &= \frac{\alpha}{2}\intomega{  \sqrt{\det g_A}  \partial_\psi^2 \transRelShapeOpt{\psi}(v,w)}
\end{align*}
where
\begin{align*}
\partial_\psi \transRelShapeOpt{\psi} (v)=& 2\sum_{m=1}^{3}\left(g_A^{-1}D^2\psi^mg_A^{-1}\right):D^2v^m\\ &-2\left(D^2v \cdot n[\psi] + D^2\psi \cdot \partial_\psi n[\psi](v) \right) :\left(g_A^{-1} \secFF_Ag_A^{-1}\right)\,,\quad\\
\partial_\psi^2 \transRelShapeOpt{\psi}(v,w) =&2\sum_{m=1}^{3}\left(g_A^{-1}D^2w^mg_A^{-1}\right):D^2v^m\\ &-2 \left(D^2v \cdot \partial_\psi n[\psi](w) + D^2w \cdot \partial_\psi n[\psi](v) \right) :\left(g_A^{-1} \secFF_Ag_A^{-1}\right)\\
&-2 \left( D^2\psi \cdot \partial_\psi^2 n[\psi](v,w) \right) :\left(g_A^{-1} \secFF_Ag_A^{-1}\right)\,.
\end{align*}
To compute the first and second variation of the normal field $n[\psi]$,
we recall the definition of the metric $g[\psi] = (\nabla \psi)^\top\nabla \psi$
in the deformed configuration and observe that $|n[\psi]|^2 = 1$ implies $0 = \partial_\psi (|n[\psi]|^2)(v) = 2 n[\psi] \cdot \partial_\psi n[\psi](v)$. 
Hence, there exist $\alpha_1,\, \alpha_2\in \R$ s.t. $\partial_\psi n[\psi](v) = \alpha_1 \partial_1 \psi + \alpha_2 \partial_2 \psi$ and therefore
$\partial_\psi n[\psi](v)\cdot \partial_k \psi = \alpha_1 g[\psi]_{{k1}} +\alpha_2 g[\psi]_{{k2}}$.
Furthermore, $0= \partial_\psi (n[\psi] \cdot \partial_k \psi)(v) = \partial_\psi n[\psi](v)\cdot \partial_k \psi + n[\psi] \cdot \partial_k v$ implies
$g[\psi] (\alpha_1,\alpha_2)^\top =\nabla \psi^\top \partial_\psi n[\psi](v)= - \nabla v^\top n[\psi]$
and thus $(\alpha_1,\alpha_2)^\top = - g[\psi]^{-1} \nabla v^\top n[\psi]$\,. 
Finally, for the first variation of $n[\phi]$, we obtain
\begin{align*}
\partial_\psi n[\psi](v) = - \nabla \psi g[\psi]^{-1} \nabla v^\top n[\psi].
\end{align*}
For the second variation we obtain
\begin{align*}
\partial_\psi^2 n[\psi](v,w) = &- \nabla w g[\psi]^{-1} \nabla v^\top n[\psi] - \nabla \psi \partial_\psi \left(g[\psi]^{-1}\right)(w) \nabla v^\top n[\psi]\\ &- \nabla \psi g[\psi]^{-1} \nabla v^\top \partial_\psi n[\psi](w).
\end{align*}
where $\partial_\psi \left(g[\psi]^{-1}\right)(w)$ can be evaluated taking into account
\begin{align*}
0 = \partial_\psi \left(g[\psi]^{-1} g[\psi]\right)(w) = \partial_\psi \left(g[\psi]^{-1}\right)(w) g[\psi] + g[\psi]^{-1} \partial_\psi\left(g[\psi]\right)(w)
\end{align*}
and $\partial_\psi\left(g[\psi]\right)(w) = (\nabla w)^\top \nabla \psi+ (\nabla \psi)^\top \nabla w$, which implies
\begin{align*}
\partial_\psi \left(g[\psi]^{-1}\right)(w) =& - g[\psi]^{-1} \partial_\psi\left(g[\psi]\right)(w) g[\psi]^{-1}\\ =& -g[\psi]^{-1} \left( \nabla w^\top \nabla \psi + \nabla \psi^\top \nabla w \right)g[\psi]^{-1}\,.
\end{align*}
Based on this, we straightforwardly obtain
\begin{align*}
\partial_\psi G[\psi](p)(v) &=\intomega{\left((\nabla v)^\top \nabla \psi + (\nabla \psi)^\top \nabla v\right) : p}\,, \quad\\
\partial_\psi^2 G[\psi](p)(v,w) &=\intomega{\left((\nabla v)^\top \nabla w + (\nabla w)^\top \nabla v \right) : p}\,.
\end{align*}

We remark that a proof of convergence of the second order method would require invertibility of the Hessian $D^2 L$, which we have always obtained in our numerical computations.
However, note that the Hessian $D^2 E$ is in general not invertible.

Finally, note that an algorithmic generalization on multiple charts would be straightforward.
E.g. for two DKT charts $\psi_B^1, \psi_B^2$ of the deformed configuration corresponding to given DKT charts $\phi_A^1, \phi_A^1$,
which share degrees of freedom on the common boundary
$\mathcal{S} = \partial \mathcal{M}_A^1 \cap \partial \mathcal{M}_A^2$,
we require consistency of the DKT degrees of freedom, \ie $\psi_B^1(z) =  \psi_B^2(z)$ and $\nabla \psi_B^1(z) = \nabla \psi_B^2(z)$ for all $z \in \mathcal{S}_h$.

\section{Numerical Results} \label{sec:results}
In this section, the presented method is applied for specific choices of $\psi_A$, $f$ and $\omega$ and for $\alpha = \frac{1}{12}$.
In all our examples, we consider a sequence of triangulations on $\omega$, generated by uniform, regular (so called red) refinement starting from a coarse rectangular mesh with each rectangular cell subdivided into two triangles.
We use $\psi_A$ as initialization for $\psi_B$ on the coarsest mesh.
On a refined mesh, we use a prolongation of the result on the previous coarser mesh as an initialization.
In the first three examples, the surfaces are parametrized over the unit square $\omega = (0,1)\times (0,1)$ and the part of the boundary for the clamped boundary condition is set to $\Gamma_D = \{ 0 \} \times [0,1]$.
Furthermore, we will also consider an L-shaped parameter domain. Finally, an example with modified boundary conditions is shown.
\medskip

\noindent \textbf{(1) Square-shaped plate.} 
In the first experiment, the undeformed surface is a flat unit square in $\R^3$ with $\psi_A(x_1,x_2) = (x_1,x_2,0)^\top$ and $f(x_1,x_2) = (0,0,-0.1)^\top$.
Thus, $a_i^{kj} = 0$ for all $i \in \{ 0,1,2\}$ and $k,j\in \{1,2\}$.
Note that this flat case is already covered by Bartels~\cite{Ba11}.
However, as mentioned above, our numerical method differs by the enforcing of a nodal-wise metric constraint as in \cite{HoRuSi20}
instead of the linearization of the contraint in a gradient descent.
In~\ref{tab:plate}, for decreasing grid size $h$, the minimal discrete energy, the isometry error in $L^1$, the $L^1$-norm of the discrete Gauss-curvature $K_h[\psi_h] = \det(g[\psi_h]^{-1} \nabla \theta_h[\psi_h]\cdot n[\psi_h])$ with 
$\nabla \theta_h[\psi_h]\cdot n[\psi_h] = \left(\sum_{l=1}^3 n_l[\psi_h] \partial_k\theta_h^{j}[\psi_h^l]\right)_{k,j=1,2}$
and the $L^2$ approximate error in the Hessian of the energy are shown. 
Since we do not know the minimizer of this problem explicitly,
we compare the discrete Hessian of the discrete minimizer for grid size $h$ 
to the discrete Hessian of the finest solution with grid size $h^\ast = 0.0014$.
For a numerical quadrature, we prolongate functions on to the finest mesh.
We obtain an approximate linear convergence rate  for $\nabla \theta[\psi_h]$. 
This rate coincides with the rate for the DKT interpolation on $H^3(\omega,\R^3)$ as stated in~Section~\ref{sec:discrete}.
Note that this is the same convergence rates as obtained for the linearized gradient flow scheme in \cite{Ba11}.
Furthermore, the convergence of the isometry error is of second order, whereas theoretically we can only guarantee a linear convergence rate.
Compare here the results in Table~\ref{tab:saddleBdry}.
By Gauss' theorema egregium, a smooth surface isometric to the plate has a vanishing Gaussian curvature. 
Here, we observe that $K_h$ indeed approaches zero, with approximately linear order of convergence.
\begin{table}[htbp]
	\centering
	\begin{tabularx}{0.8\textwidth}{c c c c c}
		\toprule
		$h$
		& $E_h[\psi_h]$
		& $\norm{g[\psi_h] - g_A}_{L^1}$
		& $\norm{K_h[\psi_h] }_{L^1}$
		& $\norm{ \nabla \theta_h [\psi_h] - \nabla \theta_{h^\ast} [\psi_{h^\ast}]}_{L^2}$\\ 
		\midrule
		0.0442&0.00595312& 5.332e-05 & 0.0009403 &  0.00899 \\
		0.0221&0.00595271& 1.329e-05 & 0.0002892 &  0.003455  \\
		0.0111&0.00595210& 3.324e-06 & 0.000140  &  0.001711  \\
		0.0055&0.00595195& 8.310e-07 & 6.928e-05 &  0.000847  \\
		0.0028& 0.00595191& 2.078e-07 &3.438e-05 &  0.0004139  \\
		0.0014&0.00595190 & 5.194e-08 &1.712e-05 & -  \\
		\bottomrule
	\end{tabularx}
	\caption{Experimental convergence evaluation for example (1):  
	grid size, discrete energy, isometry error in $L^1$, $L^1$ 
	norm of the discrete Gaussian curvature, and approximate $L^2$ error for the hessian.}
	\label{tab:plate}
\end{table}
\medskip

\noindent \textbf{(2) Half Cylinder.} 
In the second experiment we consider
$$ \psi_A: [0,1]^2 \to \R^3; \ \psi_A(x_1,x_2) = \left(\pi^{-1} \sin ( \pi x_1 ), x_2, \pi^{-1} \cos ( \pi x_1)\right)^\top\,, $$
which isometrically parametrizes a half-cylinder as the undeformed configuration and apply the loads
\begin{align*}
	f_1(x_1,x_2) = (0,1,0)^\top \,,
	\quad f_2(x_1,x_2) = \begin{cases}
		(-8,1,0)^\top & \text{if } (x_1,x_2)^\top \in [0,\frac{1}{2}]\times [\frac{1}{2},1]\,, \\
		(0,1,0)^\top  &  \text{else}\, .
	\end{cases}
\end{align*}
In~\ref{tab:cyl}, we list the same quantities as for example (1), now for both loads $f_1$ and $f_2$.
Since the half cylinder is isometric to the plate, an isometric deformation of it should also have vanishing Gaussian curvature.
Here, we observe a less than linear experimental rate of convergence of the discrete Gaussian curvature, 
whereas the convergence of the discrete Hessian again appears to be linear. 
In~Figure~\ref{fig:cylDef}, the undeformed cylinder parametrized by $\psi_A$,
and the different discrete deformations of the half cylinder due to the two different loads are displayed from different perspectives for the numerical results on the finest grid size. 
Here, the elongated box attached to the surfaces illustrates the clamped boundary condition.
\begin{figure}[htbp]
	\resizebox{0.95\textwidth}{!}{
	\begin{tikzpicture}

			\node[inner sep=0pt] (Undeformed) at (5.2,1.65) {\includegraphics[width= 0.006\linewidth]{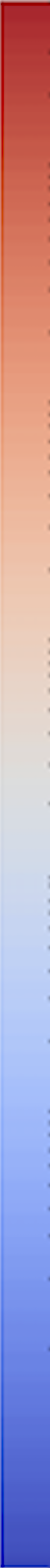}};
			\node[scale=0.7] at (5.8,3.25) {$1.1e-06$};
			\node[scale=0.7] at (5.8,1.5) {$2.0e-08$};
			\node[scale=0.7] at (5.8,0) {$6.7e-10$};
			\node[inner sep=0pt] (Undeformed) at (3.41,1.65) {\includegraphics[width= 0.18\linewidth]{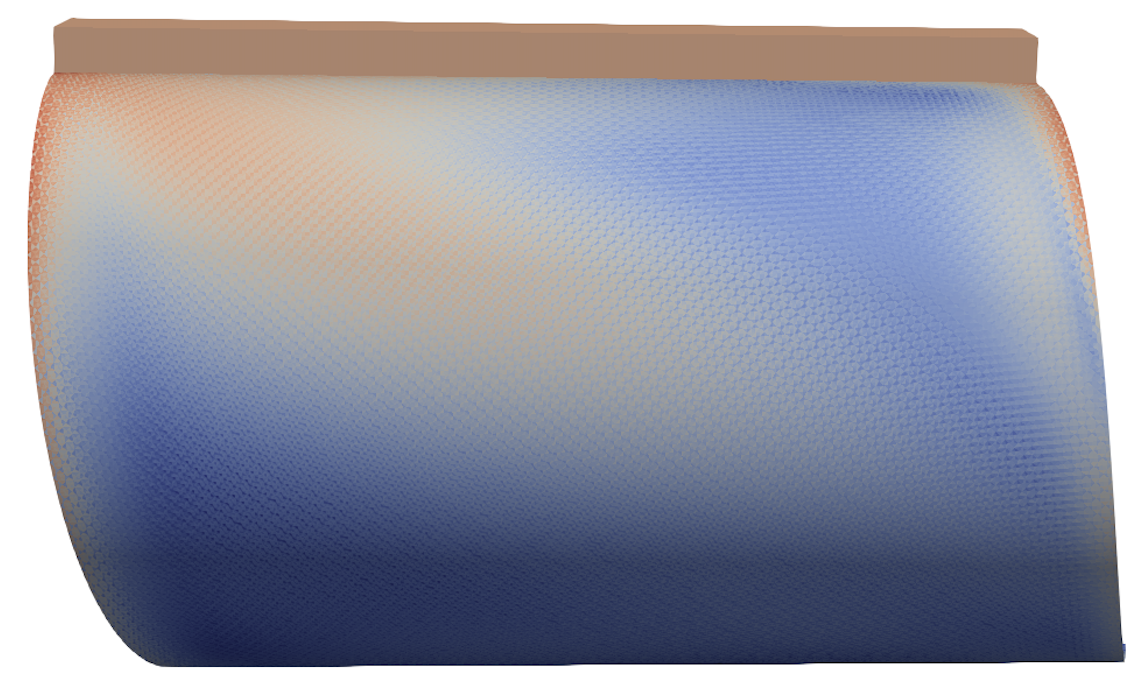}};

			\node[inner sep=0pt] (Undeformed) at (0.23,1.75) {\includegraphics[width= 0.18\linewidth]{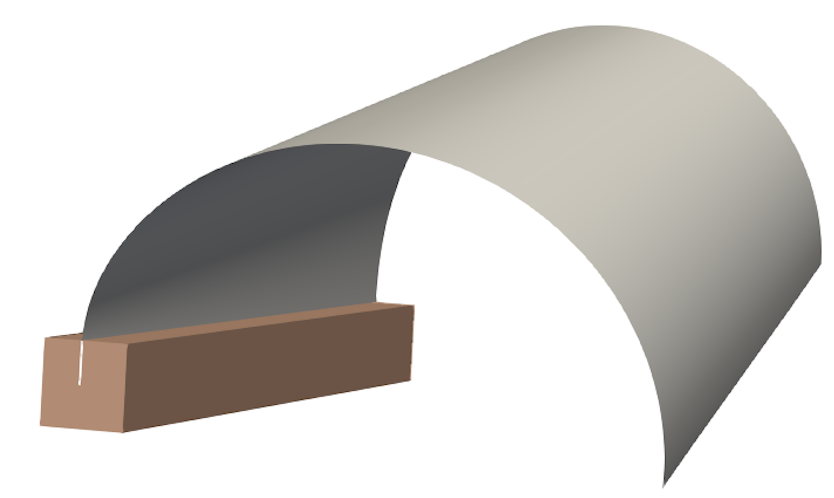}};
			\node[inner sep=0pt] (Undeformed) at (-3.0,1.75) {\includegraphics[width= 0.18\linewidth]{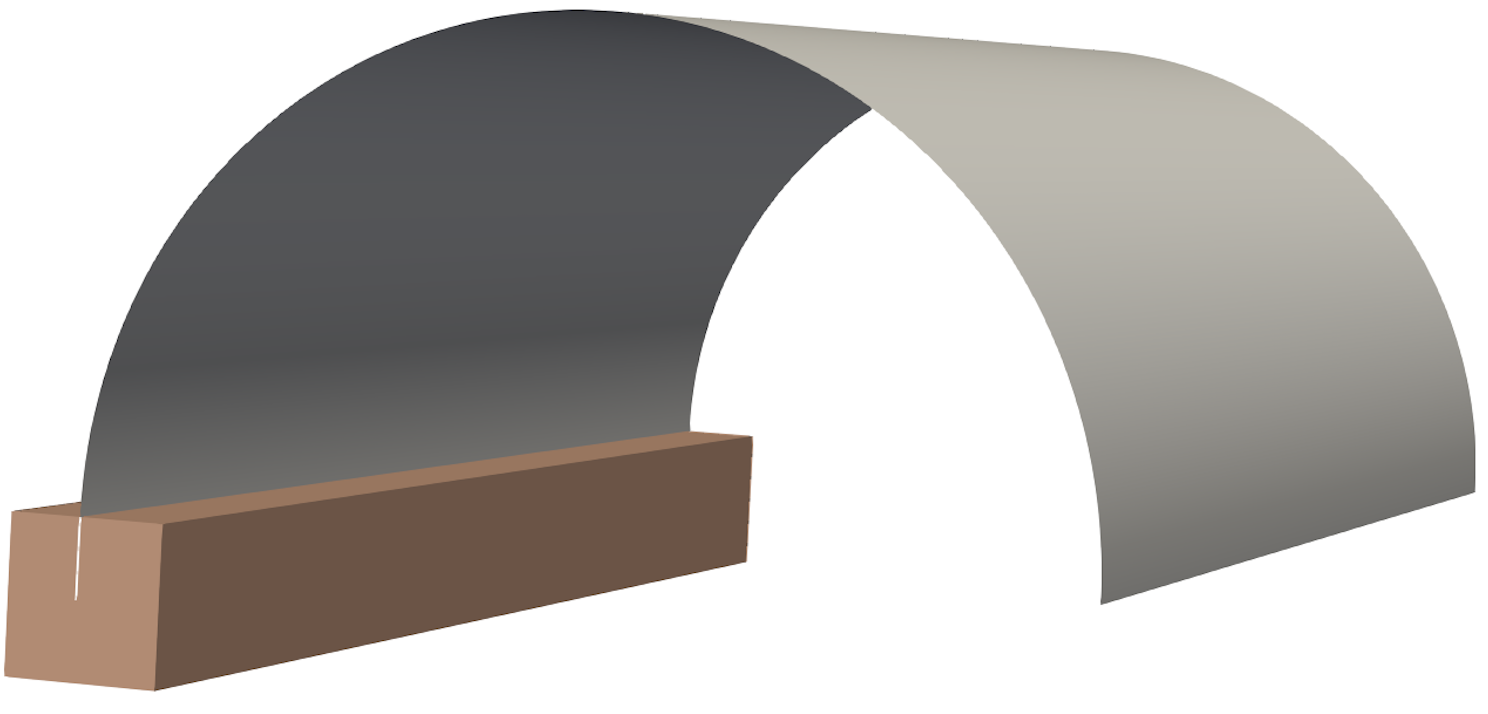}};
			\node[inner sep=0pt] (Undeformed) at (7.95,0.9) {\includegraphics[width= 0.18\linewidth]{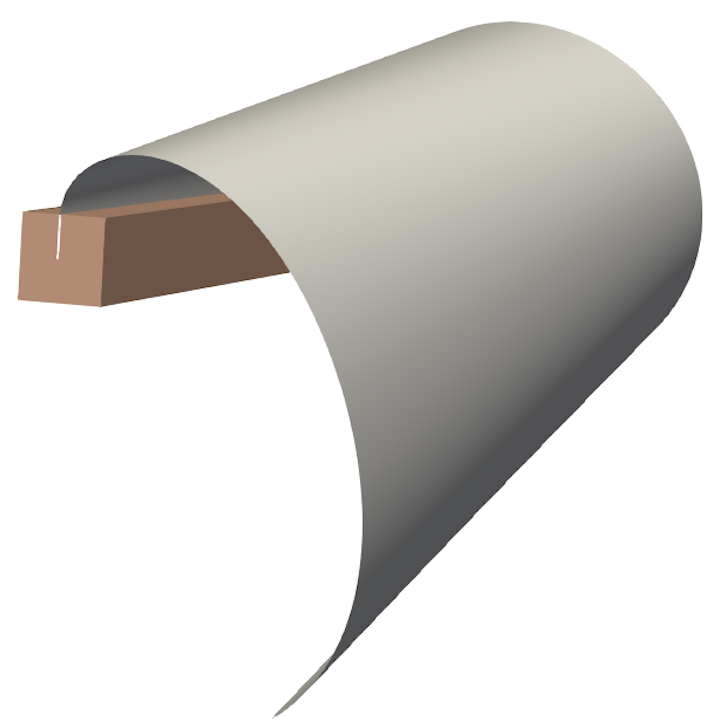}};
			\node[inner sep=0pt] (Undeformed) at (10.6,1.65) {\includegraphics[width= 0.1\linewidth]{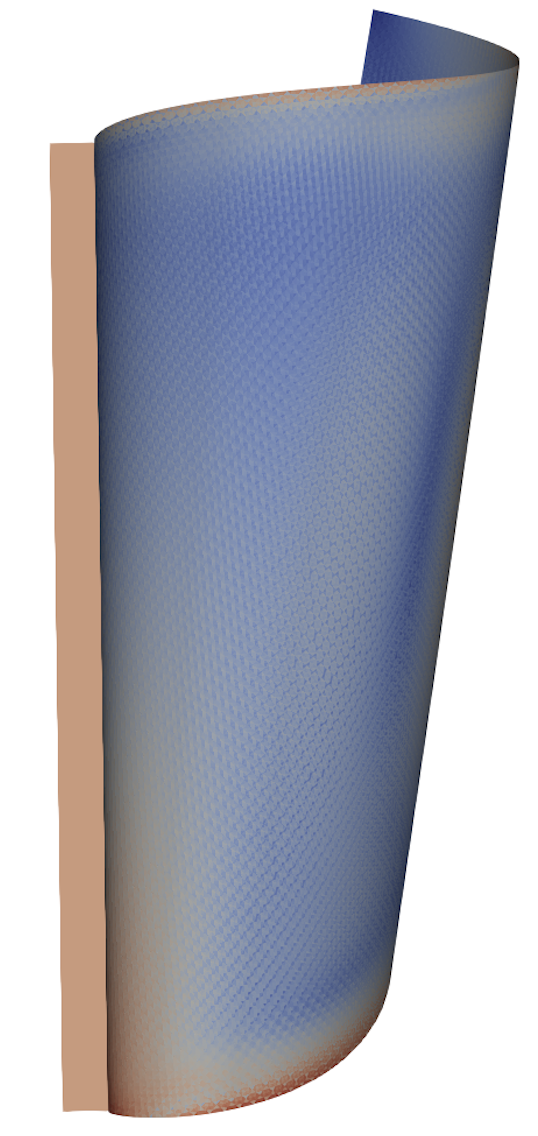}};
						\node[inner sep=0pt] (Undeformed) at (11.7,1.65) {\includegraphics[width= 0.006\linewidth]{images/skala.png}};
			\node[scale=0.7] at (12.3,3.25) {$9.1e-06$};
			\node[scale=0.7] at (12.3,1.45) {$1.0e-07$};
			\node[scale=0.7] at (12.3,0) {$2.8e-09$};

			\draw (-1.35,-0.8) -- (-1.35,3.5);
			\draw (6.4,-0.8) -- (6.4,3.5);
	\end{tikzpicture}}
	\caption{Left: Undeformed configuration for example (2).
	Middle: Deformed configuration for load $f_1$ and color-coded corresponding to an element-wise evaluation of 
	$\norm{ \nabla \theta_h [\psi_h] - \nabla \theta_{h^*} [\psi_{h^\ast}]}_{L^2(T)}$ for $h = 0.0028, h^\ast = 0.0014$ with associated deformation $\psi_{h^\ast}$ using logarithmic scaling. 
	Right: the same for load $f_2$.  }
	\label{fig:cylDef}
\end{figure}

\begin{table}[htbp]
	\centering
\begin{tabularx}{0.72\textwidth}{c c c c c c c}
	\toprule
	$h$
	& \multicolumn{2}{c}{$E_h[\psi_h]$}
	& \multicolumn{2}{c}{$\norm{K_h[\psi_h] }_{L^1}$}
	& \multicolumn{2}{c}{$\norm{ \nabla \theta_h [\psi_h] - \nabla \theta_{h^\ast} [\psi_{h^\ast}]}_{L^2}$}\\
	\cmidrule(lr){2-3}\cmidrule(lr){4-5} \cmidrule(lr){6-7}
	&$f_1$ & $f_2$ & $f_1$ & $f_2$ & $f_1$ & $f_2$\\
	\midrule
	0.0442&   0.0416  & 1.0628 & 0.1164 & 0.4478 & 0.2927 &  0.7181\\
	0.0221&   0.0386  & 1.0427 & 0.0703 & 0.2545 & 0.1497 &  0.3590\\
	0.0111&   0.0377  & 1.0367 & 0.0413 & 0.1537 & 0.0854 &  0.2064\\
	0.0055&   0.0376  & 1.0346 & 0.0249 & 0.0906 & 0.0489 &  0.1205\\
	0.0028&   0.0373  & 1.0338 & 0.0150 & 0.0523 & 0.0255 &  0.0625\\
	0.0014&   0.0372  & 1.0335 & 0.0093 & 0.0309 & -      &  -\\
	\bottomrule
\end{tabularx}
\caption{Experimental convergence evaluation for example (2): grid size, discrete energy, $L^1$ norm of the discrete Gaussian curvature and approximate $L^2$ error in the hessian for loads $f_1$ and $f_2$.}
\label{tab:cyl}
\end{table}

\medskip

\noindent \textbf{(3) Saddle-shaped surface.}
We consider a saddle-shaped surface 
as reference configuration parametrized via
\begin{align}\label{eq:expsiA}
\psi_A(x_1,x_2) &= \left( x_1, x_2, \tfrac{1}{2}\left( (x_1 -\tfrac{1}{2} )^2 -(x_2 - \tfrac{1}{2})^2 \right) \right)^\top \, .
\end{align}
over the unit square. Obviously, $\psi_A$ is no isometric deformation of $\omega$.
Figure~Figure~\ref{fig:defSaddle} shows the undeformed saddle and two different deformed configurations
for $f_1(x_1,x_2) = (0,0,-1)^\top$ and $f_2(x_1,x_2) =(0,0,-2.5)^\top$, respectively.
In~\ref{tab:saddle}, the discrete energies, and the experimental convergence of the discrete
Gaussian curvature and the discrete Hessian for decreasing grid size $h$ are shown for both forces.
As approximate ground truth, we consider again the evaluation on the finest grid.
Different to the first two examples, where the reference configurations are isometric to a planar domain (the plate and the half cylinder),
we observe a less than linear experimental order of convergence, both for the Gaussian curvature, and for the discrete Hessian.
Here, let us recall that \ref{theoapprox} only applies for functions which can be approximated by smooth isometries.
In fact, we can only guarantee $H^2$ regularity for a minimizer $\psi_B$ due to the lower bound for the continuous energy.
However, for the estimate~\eqref{eq:uniform}, we require an approximation of $\psi_B$ in $H^3$ which is isometric to $\psi_A$.
This approximation result was proven by Hornung~\cite{Ho11} in the flat case, where he essentially made use of the property that smooth isometries are developable. On this background a generalization of Hornung's result remains unclear in the general case of curved surfaces.
Here, we actually need the smooth approximation property as an additional assumption.
Furthermore, the dependence of $h$ on $\epsilon$ and the $H^3$ norm of the approximation in~\eqref{eq:heps} impacts the resulting convergence rate.

\begin{figure}[htbp]
	\centering
	\resizebox{0.95\textwidth}{!}{
	\begin{tikzpicture}
		\node[inner sep=0pt] (Undeformed) at (2.8,1.4) {\includegraphics[width=0.15\textwidth]{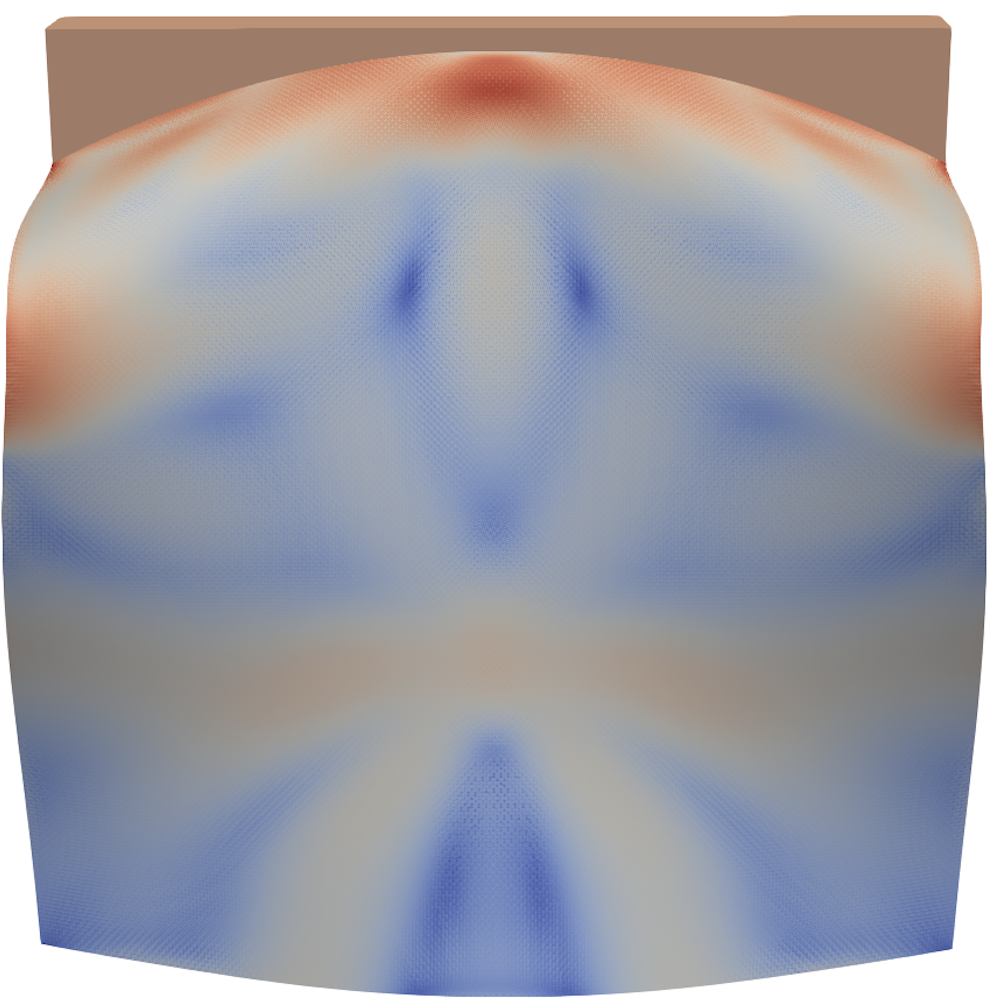}};
				\node[inner sep=0pt] (Undeformed) at (4.35,1.45) {\includegraphics[width= 0.006\linewidth]{images/skala.png}};
		\node[scale=0.7] at (4.95,3.05) {$4.1e-04$};
		\node[scale=0.7] at (4.95,1.5) {$5.0e-06$};
		\node[scale=0.7] at (4.95,-0.2) {$3.0e-08$};

		\node[inner sep=0pt] (Undeformed) at (-3.9,1.4) {\includegraphics[width= 0.13\linewidth]{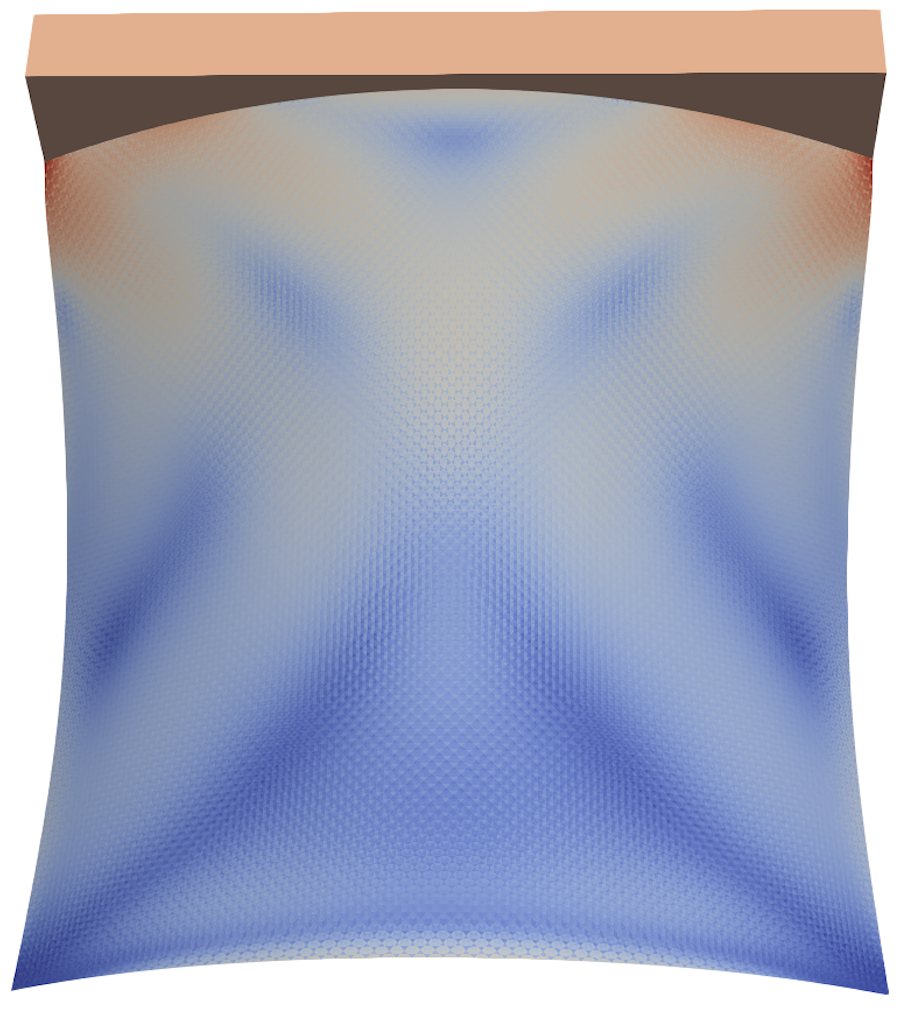}};
				\node[inner sep=0pt] (Undeformed) at (-2.5,1.45) {\includegraphics[width= 0.006\linewidth]{images/skala.png}};
		\node[scale=0.7] at (-1.9,3.05) {$2.4e-04$};
		\node[scale=0.7] at (-1.9,1.5) {$8.0e-07$};
		\node[scale=0.7] at (-1.9,-0.2) {$2.2e-09$};

		\node[inner sep=0pt] (Undeformed) at (-6.6,1.75) {\includegraphics[width= 0.18\linewidth]{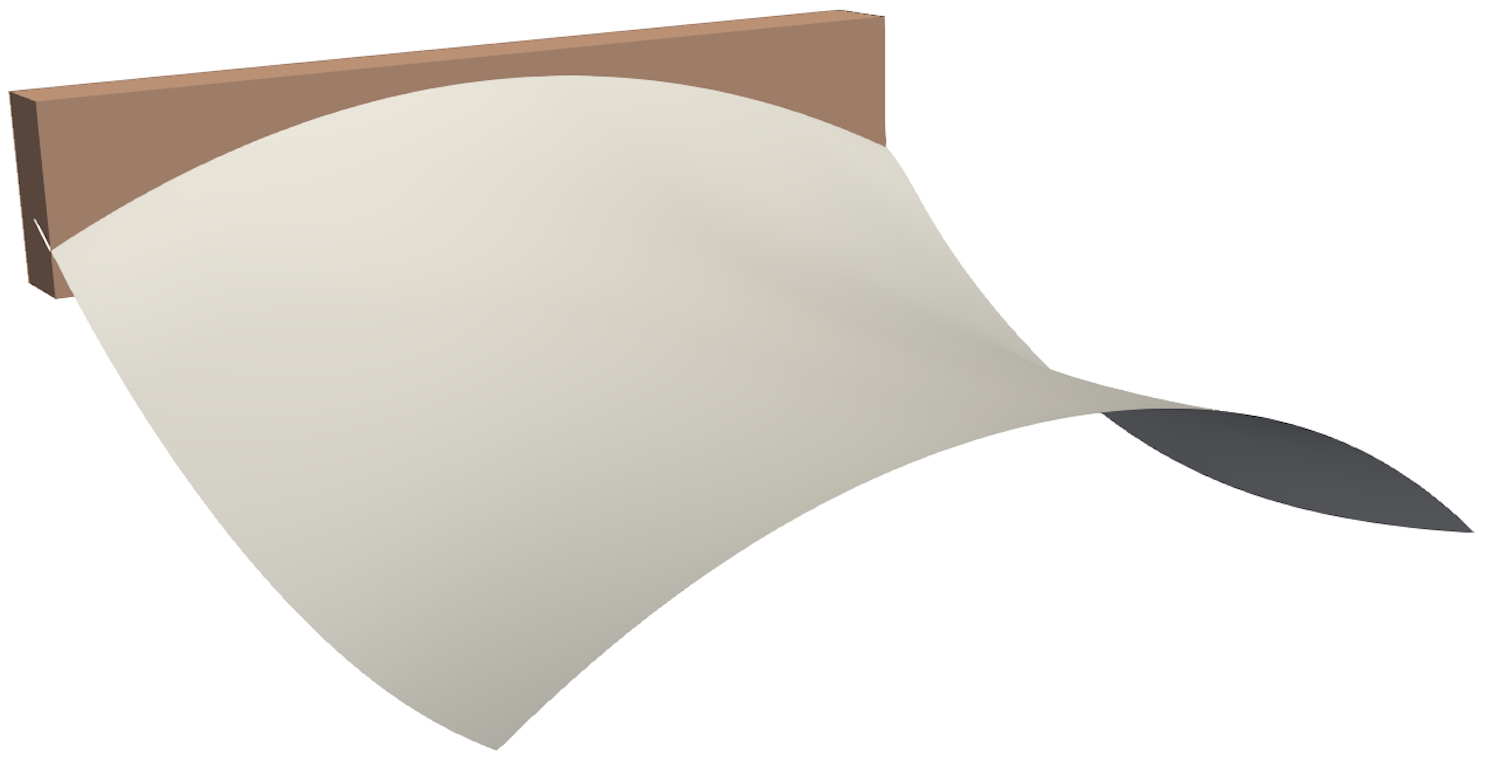}};

		\node[inner sep=0pt] (Undeformed) at (-10,1.9){\includegraphics[width=0.18\linewidth]{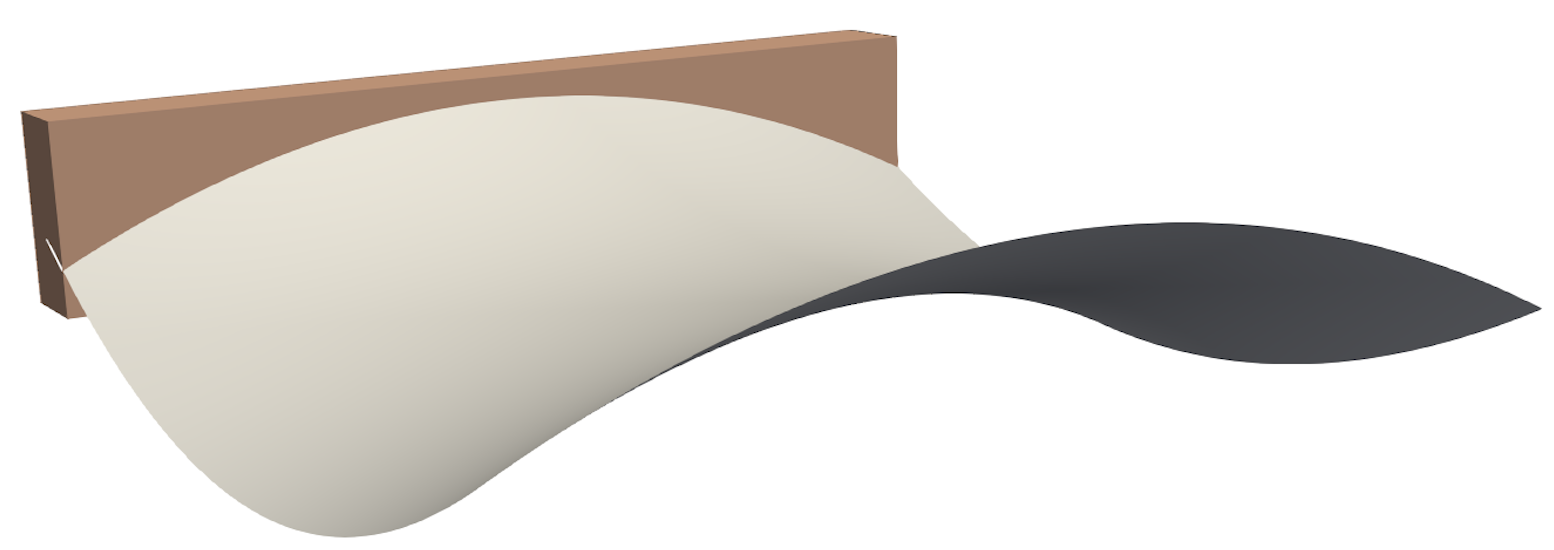}};

		\node[inner sep=0pt] (Undeformed) at (0.05,1.4) {\includegraphics[width=0.15\linewidth]{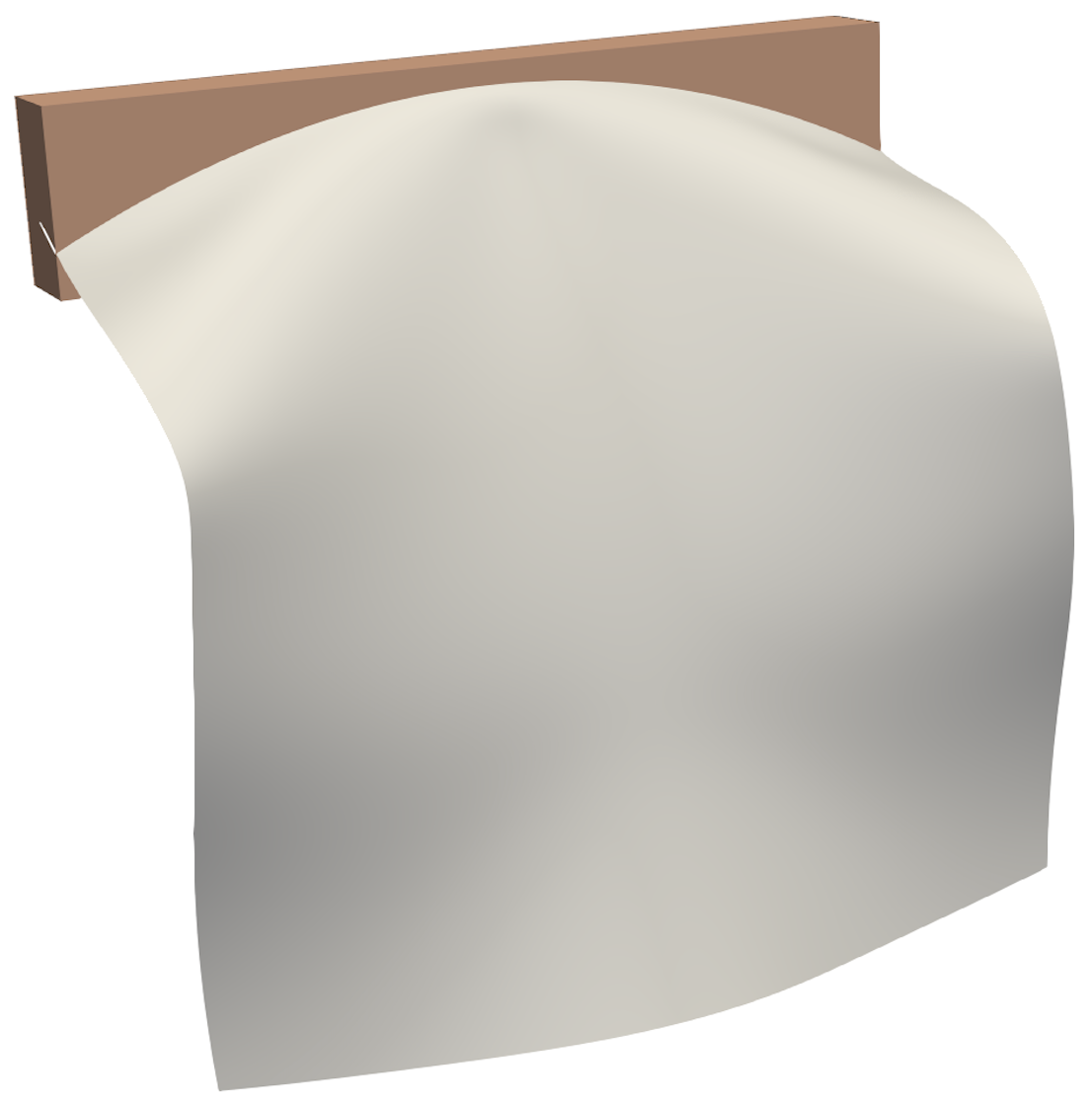}};
		\draw (-1.25,3.2) -- (-1.25,-0.3);
		\draw (-8.3,3.2) -- (-8.3,-0.3);
	\end{tikzpicture}}
	\caption{Left: Undeformed configuration for example (3).
		Middle: Deformed configuration for load $f_1$ and colorcoded corresponding to an element-wise evaluation of $\norm{ \nabla \theta_h [\psi_h] - \nabla \theta_{h^\ast} [\psi_{h^\ast}]}_{L^2(T)}$ for $h = 0.0028, h^\ast = 0.0014$ using logarithmic scaling. Right: the same for load $f_2$.}
	\label{fig:defSaddle}
\end{figure}
\begin{table}[htbp]
	\begin{tabularx}{0.83\textwidth}{c c c c c c c}
		\toprule
		$h$
		& \multicolumn{2}{c}{$E_h[\psi_h]$}
		& \multicolumn{2}{c}{$\norm{K_h[\psi_h] - K_{h^\ast}[\psi_{h^\ast}]}_{L^1(\omega)}$}
		& \multicolumn{2}{c}{$\norm{ \nabla \theta_h [\psi_h] - \nabla \theta_{h^\ast} [\psi_{h^\ast}]}_{L^2(\omega)}$}\\
		\cmidrule(lr){2-3}\cmidrule(lr){4-5} \cmidrule(l){6-7}
		&$f_1$ & $f_2$& $f_1$ & $f_2$& $f_1$ & $f_2$ \\
		\midrule
		0.0442&0.3236&1.2792&0.4212&0.9486&1.5903&3.5378 \\
		0.0221&0.2442&1.1052&0.2666&0.9135&1.0699&2.5771  \\
		0.0111&0.2136&1.0062&0.1989&0.7499&0.8182&2.3427   \\
		0.0055&0.1975&0.9190&0.1212&0.5657&0.6336&2.0894 \\
		0.0028&0.1843&0.8203&0.0520&0.3510&0.4046&1.5613  \\
		0.0014&0.1745&0.7077&-&-&-&-\\
		\bottomrule
	\end{tabularx}
	\caption{Experimental convergence evaluation for example (3): grid size, discrete energy, isometry error in $L^1$, approximate $L^1$ error of the discrete Gaussian curvature and approximate $L^2$ error in the hessian for loads $f_1$ and $f_2$.}
	\label{tab:saddle}
\end{table}

\medskip

\noindent \textbf{(4) L-shaped saddle-shaped surface.}
In this example, the parameter domain is given by an L-shape 
$\omega = (0,\frac{1}{2})\times(0,\frac{1}{2}) \cup (0,1)\times (\frac{1}{2}, 1)$.
We consider the initial configuration parametrized by $\psi_A$ as defined in~\eqref{eq:expsiA}.
The clamped boundary conditions are enforced on $\Gamma_D = \{1\}\times [\frac{1}{2},1]$ and two different loads
\begin{align*}
	f_1(x_1,x_2) = (0,0,-1)^\top\,,
	\quad f_2(x_1,x_2) =
	\begin{cases}
		(-5,0,-2)^\top  & \text{if } (x_1,x_2) \in [0,\frac{1}{2}]\times [0,\frac{1}{2}]\,, \\
		(0,0,-2)^\top   & \text{else}\,.
	\end{cases}
\end{align*}
are applied.
In~Figure~\ref{fig:defLshape}, the undeformed and the  deformed configurations are shown.
\begin{figure}[htbp]
	\resizebox{0.95\textwidth}{!}{
	\begin{tikzpicture}
			\node[inner sep=0pt] (Undeformed) at (-1.5,1.75) {\includegraphics[width=0.2\linewidth]{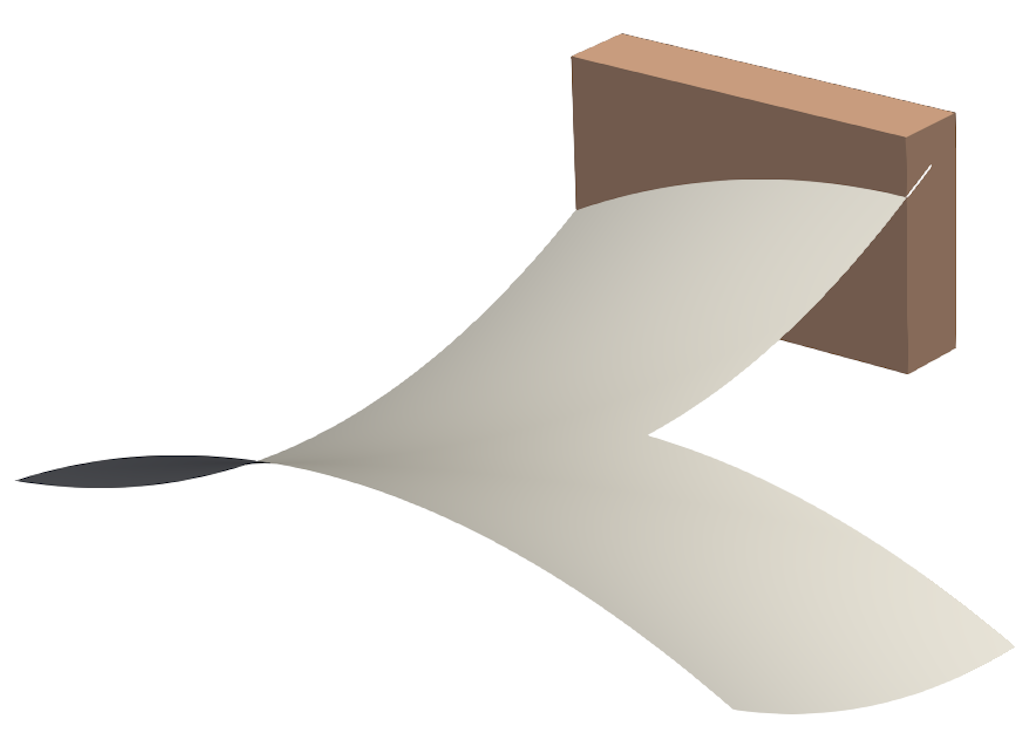}};
			\node[inner sep=0pt] (Undeformed) at (4.55,1.15) {\includegraphics[width=0.15\linewidth]{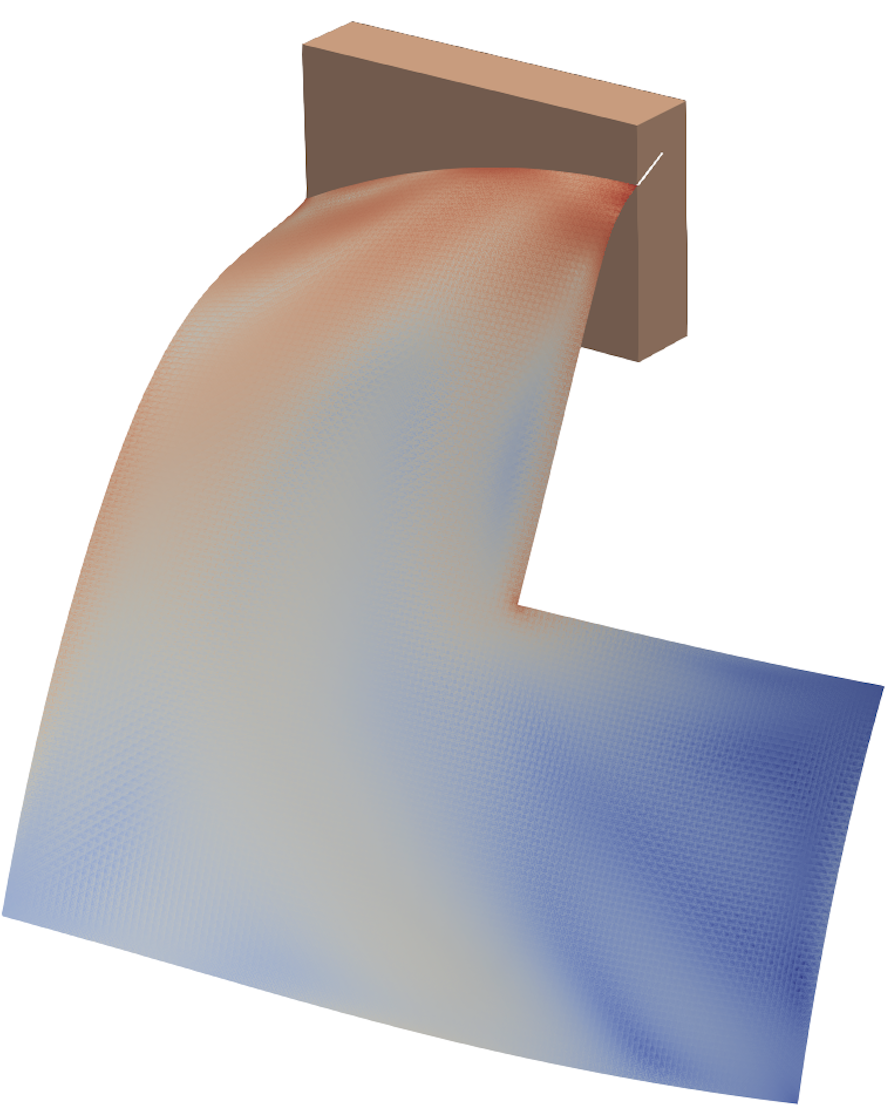}};
						\node[inner sep=0pt] (Undeformed) at (6.1,1.15) {\includegraphics[width= 0.006\linewidth]{images/skala.png}};
			\node[scale=0.7] at (6.7,2.75) {$4.9e-04$};
			\node[scale=0.7] at (6.7,1.25) {$3.0e-07$};
			\node[scale=0.7] at (6.7,-0.5) {$1.1e-10$};

			\node[inner sep=0pt] (Undeformed) at (1.55,1.2) {\includegraphics[width= 0.15\linewidth]{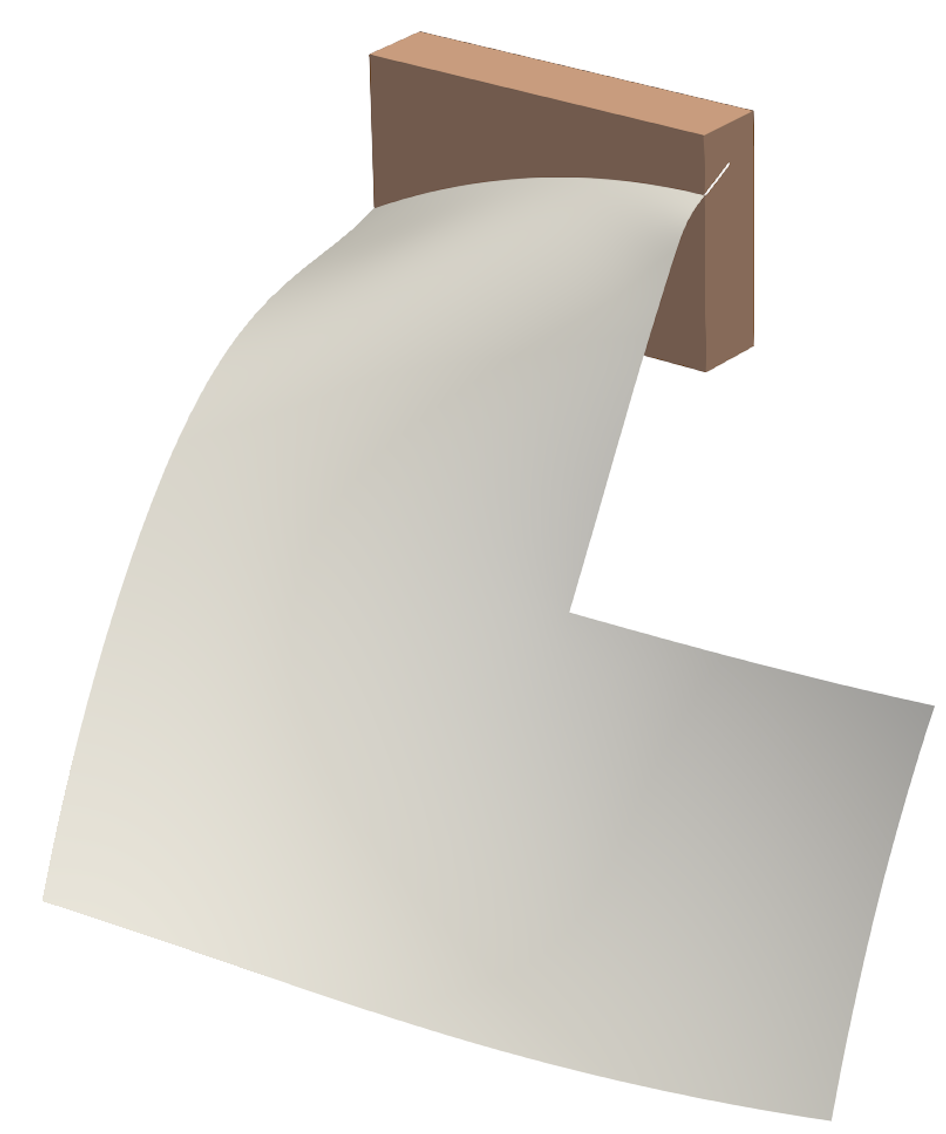}};
			\node[inner sep=0pt] (Undeformed) at (8.15,1.1) {\includegraphics[width= 0.1\linewidth]{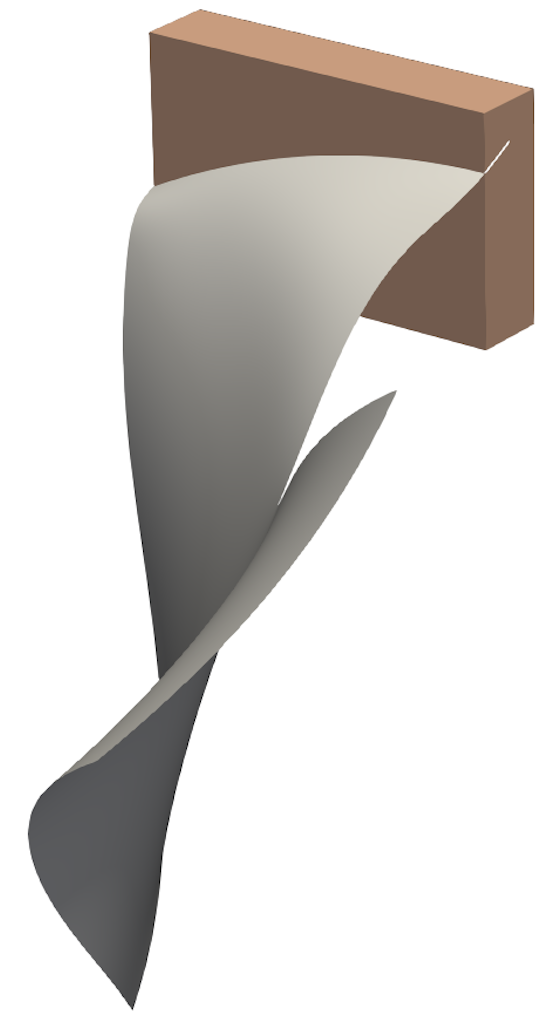}};
						\node[inner sep=0pt] (Undeformed) at (13,1.15) {\includegraphics[width= 0.006\linewidth]{images/skala.png}};
			\node[scale=0.7] at (13.6,2.75) {$2.6e-03$};
			\node[scale=0.7] at (13.6,1.25) {$1.0e-05$};
			\node[scale=0.7] at (13.6,-0.5) {$2.0e-08$};
			\node[inner sep=0pt] (Undeformed) at (11,1.1) {\includegraphics[width= 0.23\linewidth]{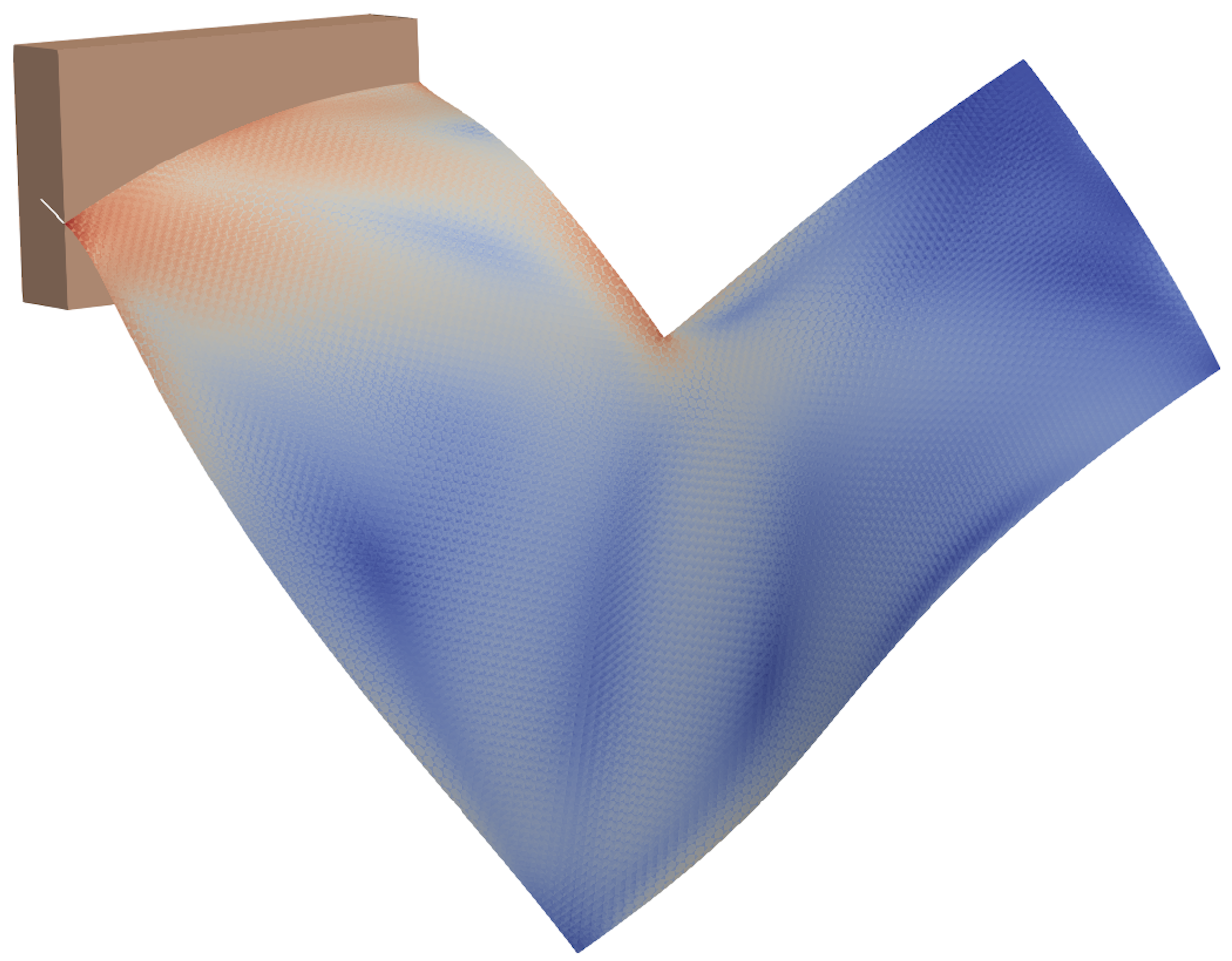}};

			\draw (0.16,-1) -- (0.16,3.1);
			\draw (7.35,-1) -- (7.35,3.1);
	\end{tikzpicture}}
	\caption{
	    Left: Undeformed configuration for example (4).
		Middle: Deformed configuration for load $f_1$ and color-coded corresponding to an element-wise evaluation of $\norm{ \nabla \theta_h [\psi_h] - \nabla \theta_{h^*} [\psi_{h^*}]}_{L^2(T)}$ for $h = 0.0028, h^* = 0.0014$ using logarithmic scaling.
		Right: the same for load $f_2$.}
	\label{fig:defLshape}
\end{figure}

\medskip

\noindent \textbf{(5) Saddle-shaped surface with different boundary conditions.}
In this last example, we consider as  in~\eqref{eq:expsiA} 
an initial configuration parametrized by $\psi_A$ over the unit square $\omega = (0,1)^2$. Instead of applying a force, we now enforce a deformation by imposing a modified clamped boundary conditions, namely $\psi_B(x) = \psi_A(x) + \frac{3}{16}$ for $x\in \{0\}\times [0,1]$, $\psi_B(x) = \psi_A(x) - \frac{3}{16}$ for $x\in \{1\}\times [0,1]$ and $\nabla\psi_B(x) =\nabla \psi_A(x) $ on $\{0\}\times [0,1] \cup \{1\}\times [0,1]$. In~Figure~\ref{fig:defSaddleBdryCond}, the undeformed and the resulting deformed configuration is shown.
In~\ref{tab:saddleBdry}, for decreasing grid size $h$, we depict the isometry error in $L^1$, the $L^1$-norm of the discrete Gauss-curvature and the $L^2$ approximate error in the Hessian of the energy.
As proven in Theorem~\ref{theoapprox}, we obtain linear convergence rate of the isometry error in $L^1$.
Note that in this case an approximation result of the admissible deformations by $H^3$ functions is unknown.
In accordance, we only obtain a sublinear convergence rate for the approximative second derivative.
\begin{figure}[htbp]
	\centering
	\resizebox{0.9\textwidth}{!}{
		\begin{tikzpicture}
		\node[inner sep=0pt] (Undeformed) at (-3.0,1.65) {\includegraphics[width=0.25\linewidth]{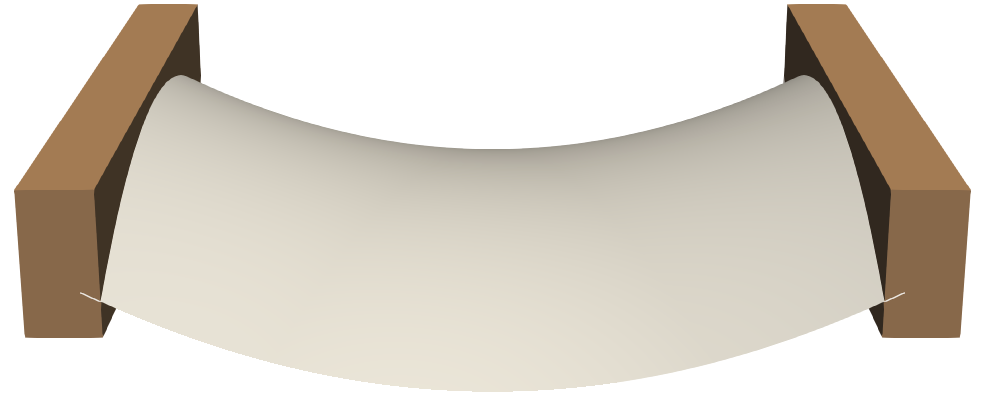}};
		\node[inner sep=0pt] (Undeformed) at (4.9,1.15) {\includegraphics[width=0.18\linewidth]{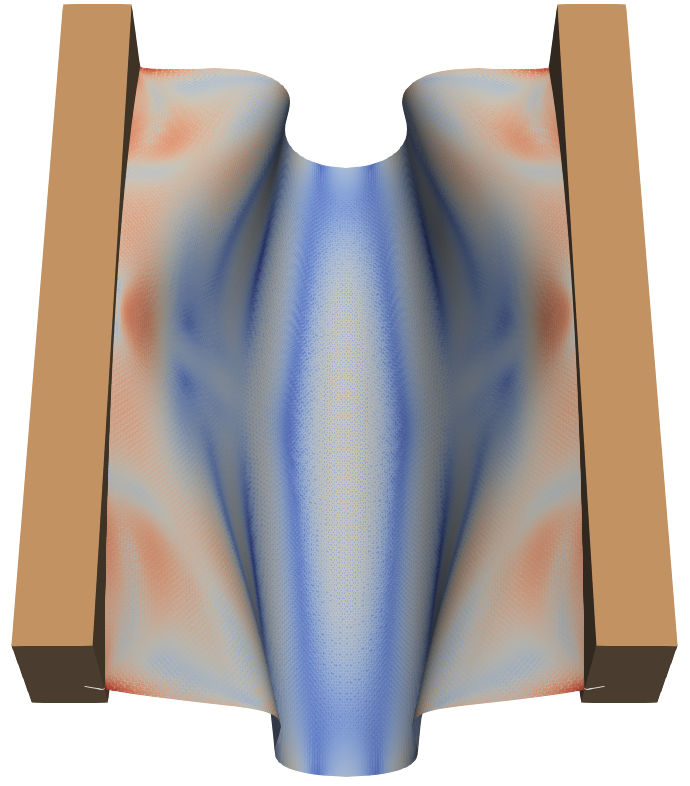}};
				\node[inner sep=0pt] (Undeformed) at (7.1,1.15) {\includegraphics[width= 0.006\linewidth]{images/skala.png}};
		\node[scale=0.7] at (7.7,2.75) {$1.3e-03$};
		\node[scale=0.7] at (7.7,1.2) {$1.0e-05$};
		\node[scale=0.7] at (7.7,-0.5) {$1.7e-07$};

		\node[inner sep=0pt] (Undeformed) at (1.55,1.2) {\includegraphics[width= 0.22\linewidth]{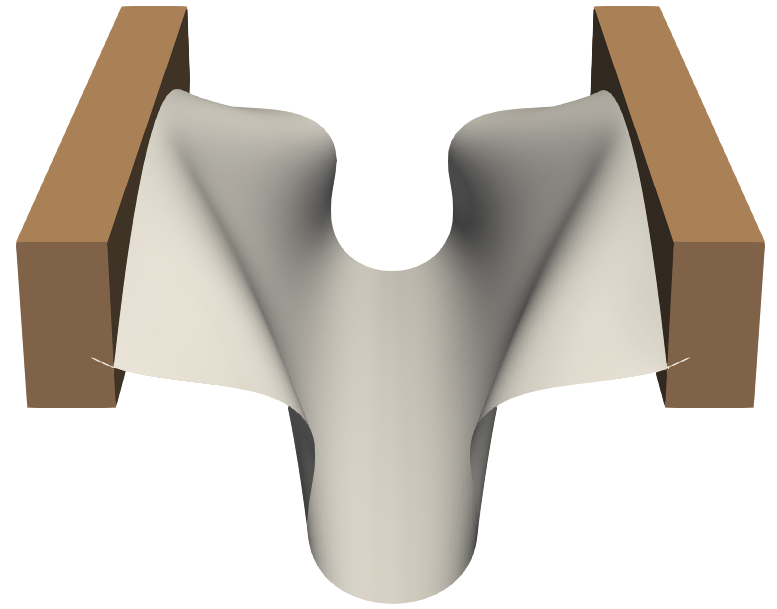}};
		
		\draw (-0.5,-0.5) -- (-0.5,3.1);
		\end{tikzpicture}}
	\caption{
		Left: undeformed configuration for example (5).
		Right: different views of the deformed configuration with prescribed boundary conditions, one image color-coded corresponding to an element-wise evaluation of $\norm{ \nabla \theta_h [\psi_h] - \nabla \theta_{h^*} [\psi_{h^*}]}_{L^2(T)}$ for $h = 0.0028, h^* = 0.0014$ using logarithmic scaling.
		}
	\label{fig:defSaddleBdryCond}
\end{figure}

\begin{table}[htbp]
	\centering
	\begin{tabularx}{0.78\textwidth}{c c c c c}
		\toprule
		$h$
		& $\norm{g[\psi_h] - g_A}_{L^1}$
		& $\norm{K_h[\psi_h] - K_{h^\ast}[\psi_{h^\ast}]}_{L^1(\omega)}$
		& $\norm{ \nabla \theta_h [\psi_h] - \nabla \theta_{h^\ast} [\psi_{h^\ast}]}_{L^2}$\\ 
		\midrule
		0.0442& 0.1025 & 2.3768 & 6.4416 \\
		0.0221& 0.0441 & 2.2603 & 5.0329 \\
		0.0111& 0.0168 & 1.8826 & 4.1574 \\
		0.0055& 0.0078 & 1.2785 & 3.4613 \\
		0.0028& 0.0041 & 0.7431 & 2.5670 \\
		0.0014& 0.0025 & -      & -  \\
		\bottomrule
	\end{tabularx}
	\caption{Experimental convergence evaluation for example (5):  
	grid size, isometry error in $L^1$, $L^1$ 
	norm of the discrete Gaussian curvature, and approximate $L^2$ error for the hessian.}
	\label{tab:saddleBdry}
\end{table}
\medskip

\bigskip

\paragraph{Acknowledgement}
We thank the anonymous reviewers for their valuable hints and helpful comments to improve this article.

\bibliographystyle{siam}

\begin{thebibliography}{10}


\bibitem{Ba11}
{\sc S.~Bartels}, {\em Approximation of large bending isometries with discrete
  {K}irchhoff triangles}, SIAM J. Numer. Anal., 51 (2013), pp.~516--525.

\bibitem{Ba15}
{\sc S.~Bartels}, {\em Numerical methods for nonlinear partial differential
  equations}, vol.~47 of Springer Series in Computational Mathematics,
  Springer, Cham, 2015.

\bibitem{Ba17}
{\sc S.~Bartels}, {\em Numerical solution of a {F}{\"o}ppl--von
  {K}\'{a}rm\'{a}n model}, SIAM Journal on Numerical Analysis, 55 (2017),
  pp.~1505--1524.

\bibitem{BaBoNo17}
{\sc S.~Bartels, A.~Bonito, and R.~H. Nochetto}, {\em Bilayer plates: Model
  reduction, {$\Gamma$}-convergent finite element approximation, and discrete
  gradient flow}, Communications on Pure and Applied Mathematics, 70 (2017),
  pp.~547--589.

\bibitem{BhLeSc16}
{\sc K.~Bhattacharya, M.~Lewicka, and M.~Sch{\"a}ffner}, {\em Plates with
  incompatible prestrain}, Archive for Rational Mechanics and Analysis, 221
  (2016), pp.~143--181.

\bibitem{BoGuNo21}
{\sc A.~Bonito, D.~Guignard, R.~Nochetto, and S.~Yang}, {\em Numerical analysis
  of the {LDG} method for large deformations of prestrained plates}, arXiv
  preprint arXiv:2106.13877,  (2021).

\bibitem{BoNoNt20}
{\sc A.~Bonito, R.~H. Nochetto, and D.~Ntogkas}, {\em {DG} approach to large
  bending plate deformations with isometry constraint}, 2020.

\bibitem{Br13}
{\sc D.~Braess}, {\em {F}inite {E}lements}, Springer, 5~ed., 2013.

\bibitem{BuClCo21}
{\sc J.~Burtscheidt, M.~Claus, S.~Conti, M.~Rumpf, J.~Sassen, and R.~Schultz},
  {\em A {P}essimistic {B}ilevel {S}tochastic {P}roblem for {E}lastic {S}hape
  {O}ptimization}.
\newblock submitted, 2021.

\bibitem{FrJaMo03}
{\sc G.~Friesecke, R.~D. James, M.~G. Mora, and S.~M{\"{u}}ller}, {\em
  Derivation of nonlinear bending theory for shells from three-dimensional
  nonlinear elasticity by {G}amma-convergence}, C. R. Math. Acad. Sci. Paris,
  336 (2003), pp.~697--702.

\bibitem{FrJaMue02b}
{\sc G.~Friesecke, R.~D. James, and S.~M{\"{u}}ller}, {\em A theorem on
  geometric rigidity and the derivation of nonlinear plate theory from
  three-dimensional elasticity}, Comm. Pure Appl. Math., 55 (2002),
  pp.~1461--1506.

\bibitem{HaNi59}
{\sc P.~Hartman and L.~Nirenberg}, {\em On spherical image maps whose
  {J}acobians do not change sign}, Amer. J. Math., 81 (1959), pp.~901--920.

\bibitem{Ho11}
{\sc P.~Hornung}, {\em Approximation of flat {$W^{2,2}$} isometric immersions
  by smooth ones}, Arch. Ration. Mech. Anal., 199 (2011), pp.~1015--1067.

\bibitem{HoRuSi20}
{\sc P.~Hornung, M.~Rumpf, and S.~Simon}, {\em On material optimisation for
  nonlinearly elastic plates and shells}, ESAIM Control Optim. Calc. Var., 26
  (2020), p.~82.

\bibitem{LeRa95}
{\sc H.~Le~Dret and A.~Raoult}, {\em The nonlinear membrane model as
  variational limit of nonlinear three-dimensional elasticity}, J. Math. Pures
  Appl. (9), 74 (1995), pp.~549--578.

\bibitem{LeRa96}
\leavevmode\vrule height 2pt depth -1.6pt width 23pt, {\em The membrane shell
  model in nonlinear elasticity: a variational asymptotic derivation}, J.
  Nonlinear Sci., 6 (1996), pp.~59--84.

\bibitem{WaBi06}
{\sc A.~W{\"a}chter and L.~T. Biegler}, {\em {On the Implementation of a
  Primal-Dual Interior Point Filter Line Search Algorithm for Large-Scale
  Nonlinear Programming}}, Mathematical Programming, 106 (2006), pp.~25--57.

\end{thebibliography}

\def\polhk#1{\setbox0=\hbox{#1}{\ooalign{\hidewidth
  \lower1.5ex\hbox{`}\hidewidth\crcr\unhbox0}}}

\end{document}